\documentclass[10pt]{amsart}
\let\Q\QQ
\newcommand{\Q}{\mathbb Q}

\input{twists.sty}
\usepackage{caption}

\usepackage{tikz,tikz-cd,pgfplots}
\pgfplotsset{compat=1.18}
\usepackage{wrapfig}
\usepackage{subcaption}
\usepackage{longtable}
\usepackage{multicol}
\usepackage{mdframed}
\usepackage{float}
\usepackage{xfrac}



\title[On algebraic vector bundles of rank $2$ over smooth affine fourfolds]{On algebraic vector bundles of rank $2$ over smooth affine fourfolds}

\author{Thomas Brazelton}
\address[T. Brazelton]{Harvard University}
\email{tbraz@math.harvard.edu}
\author{Morgan Opie}
\address[M. Opie]{Northwestern University}
\email{mpopie@northwestern.edu}
\author{Tariq Syed}
\address[T. Syed]{Heinrich-Heine-Universit\"at D\"usseldorf}
\email{tariq.syed@gmx.de}

\date{}           
\begin{document}

\begin{abstract}
    The classification of algebraic vector bundles of rank 2 over smooth affine fourfolds is a notoriously difficult problem. Isomorphism classes of such vector bundles are not uniquely determined by their Chern classes, in contrast to the situation in lower dimensions. Given a smooth affine fourfold over an algebraically closed field of characteristic not equal to $2$ or $3$, we study cohomological criteria for finiteness of the fibers of the Chern class map for rank $2$ bundles. As a consequence, we give a cohomological classification of such bundles in a number of cases. For example, if $d\leq 4$, there are precisely $d^2$ non-isomorphic algebraic vector bundles over the complement of a smooth hypersurface of degree $d$ in $\mathbb P^4_{\mathbb C}$.
\end{abstract}
\maketitle

\tableofcontents

\section{Introduction}

Suppose $k$ is an algebraically closed field and $R$ a smooth affine $k$-algebra of dimension $4$ over $k$. The main goal of this paper is to study isomorphism classes of rank $2$ finitely generated projective $R$-modules. We do this using tools from Morel--Voevodsky $\A^1$-homotopy theory, a research area that in recent years has seen the resolution of major open problems about projective modules such as Suslin's cancellation conjecture \cite{fasel2021suslins} and Murthy's corank $1$ splitting conjecture in characteristic zero \cite[Section 7.1]{Freudenthal}.

The interplay between topology and the study of finitely generated projective modules has a long history. Indeed, over affine schemes, J.-P. Serre's seminal work characterized finitely generated projective modules over commutative rings as the algebro-geometric analogue of topological vector bundles  \cite{S1}, allowing for the adaptation of a host of topological tools to the setting of projective modules. 
For example, A. Grothendieck axiomatized S.-S. Chern's eponymous invariants for complex topological vector bundles over complex manifolds, leading to the notion of Chern classes for algebraic vector bundles over smooth schemes \cite{C,G}. These characteristic classes are powerful and computable invariants in both the topological and algebraic setting, so a natural question is to what extent the Chern classes of a vector bundle determine its isomorphism class.

 Let $X$ be a smooth affine $k$-variety of dimension $d$ and $\mathcal{E}$ a rank $r$ algebraic vector bundle on $X$. The Chern classes $c_{i}(\mathcal{E}) \in \CH^{i}(X)$, $1 \leq i \leq r$, are elements of the Chow ring of $X$.  If $\mathcal{V}_{r}(X)$ denotes the set of isomorphism classes of algebraic vector bundles of rank $r$ over $X$,
the Chern classes induce a natural map
\begin{center}
$(c_{1},...,c_{r})\: \mathcal{V}_{r}(X) \rightarrow \prod_{i=1}^{r} \CH^{i}(X).$
\end{center}
By Serre's splitting theorem \cite[Th{\'e}or{\`e}me 1]{S2} and Suslin's cancellation theorem \cite[Theorem]{Su1}, any vector bundle $\mathcal{E}$ of rank $r > d$ over $X$ can be written uniquely as the direct sum of a vector bundle of rank $d$ and a trivial vector bundle over $X$. One can therefore focus on the maps with $1 \leq r \leq d$ above. 

If $r=1$ and $d$ is arbitrary, the map
\begin{center}
$c_{1}\: \mathcal{V}_{1}(X) \rightarrow \CH^{1}(X)$
\end{center}
is precisely the classical isomorphism between the Picard group and the divisor class group, which shows that the algebraic line bundles over $X$ are uniquely determined by their first Chern class. 

If $d=2$, the map
\begin{equation}\label[empty]{map:c1c2}
(c_{1},c_{2})\: \mathcal{V}_{2}(X) \rightarrow \CH^{1}(X) \times \CH^{2}(X)
\end{equation}
is also a bijection as a consequence of \cite[Theorem 1]{MS} and the relationship between algebraic $K$-theory and algebraic cycles. By continuing the study of algebraic cycles, N. Mohan Kumar and M. P. Murthy proved that if $d=3$ the map
\[(c_{1},c_{2},c_{3})\: \mathcal{V}_{3}(X) \rightarrow \CH^{1}(X) \times \CH^{2}(X)\times \CH^{3}(X)
\]
is a bijection and $(c_1,c_2)$ is at least surjective in rank $2$ \cite[Theorem 2.1]{MKM}. 
At the time, it seemed impossible with existing methods to determine if $(c_1,c_2)$ was also injective, and the question remained unresolved until the advent of $\A^1$-homotopy theory many years later. 

 F. Morel's $\mathbb{A}^1$-homotopy classification of algebraic vector bundles (cf. \cite[Theorem 7.1]{Morel}, \cite[Theorem 1]{AHW1}) provided a striking algebro-geometric analogue of 
Steenrod's homotopy classification of topological vector bundles (cf. \cite[\S 19.3]{St}) and  allowed methods of obstruction theory to be applied to algebraic vector bundles. Building on these ideas, A. Asok and J. Fasel achieved a breakthrough in \cite[Theorem 1]{AF-3fold}, proving that $(c_1,c_2)$ is indeed bijective for rank $2$ bundles if $d=3$ and the characteristic of the base field is not $2$. This completed the classification of algebraic vector bundles over smooth affine varieties over algebraically closed fields of characteristic not equal to $2$ in all dimensions less than or equal to $3$: in low dimensions at least, the isomorphism class of an algebraic vector bundle is uniquely determined by its rank and its Chern classes!

Already in the 1980s, N. Mohan Kumar had produced examples of stably trivial non-trivial algebraic vector bundles of rank $2$ over smooth affine fourfolds over algebraically closed fields \cite{MK}. 
Thus, this marks the first case in low dimensions in which algebraic vector bundles are {\em not} determined up to isomorphism by their Chern classes and rank. The classification of vector bundles of rank $2$ over smooth affine fourfolds is therefore considered an extremely difficult and subtle problem.

The main goal of this paper is to study the fibers of $(c_1,c_2)$ as in \Cref{map:c1c2}
for $X$ a smooth affine fourfold over an algebraically closed field and, in particular, to give cohomological criteria for finiteness of the fibers and for injectivity of the map. We prove:
\begin{theorem}\label{theorem-1} Let $X$ be a smooth affine variety of dimension $4$ over an algebraically closed field $k$ of characteristic not equal to $2$ or $3$. The non-empty fibers of the map \[(c_{1},c_{2}): \mathcal{V}_{2}(X) \rightarrow \CH^{1}(X) \times \CH^{2}(X)\] are singletons (resp. finite) if the motivic cohomology group $\Hmot^5(X,\Z/2(3))$ and the mod $2$ Chow group $\CH^{3}(X)/2$ are trivial (resp. finite). 
\end{theorem}
We obtain many such examples in \Cref{sec:examples}, for instance:

\begin{example}[see {\Cref{thm:p4-complement-threefold-fourfold}}] Suppose $k= \mathbb C$ and $X$ is the complement of a smooth hypersurface $D$ of degree $d \leq 4$ in $\P^4$. Then there are exactly $d^2$ isomorphism classes of algebraic vector bundles of rank $2$ over $X$, determined uniquely by their first two Chern classes in $\CH^1 (X) \times \CH^2 (X) \cong (\Z/d)^{\times2}$.
\end{example}

We prove \Cref{theorem-1} using obstruction theory involving the Moore--Postnikov factorization of the Chern class map in $\mathbb{A}^1$-homotopy theory. 
To do so, we must analyze Nisnevich cohomology of relevant twisted $\mathbb A^1$-homotopy sheaves of the
$\mathbb{A}^1$-homotopy fiber of $(c_1,c_2)$. The cohomology groups of interest can be identified with 
$\Hmot^{5}(X, \mathbb{Z}/2(3))$, $\Hnis^{3}(X,\pi^{\mathbb{A}^1}_{2}(\mathbb{A}^2 \setminus 0)(\mathcal{L}))$ and $\Hnis^{4}(X,\pi^{\mathbb{A}^1}_{3}(\mathbb{A}^2 \setminus 0)(\mathcal{L}))$, where $\mathcal L$ is a line bundle on $X$.

For the purpose of computability, we note that the group $\Hmot^{5}(X, \mathbb{Z}/2(3))$
 is isomorphic to the Nisnevich cohomology group 
 $\Hnis^{2}(X,\I^{3} (\mathcal{L}))$, where $\I^3$ denotes the sheaf associated to the third power of the fundamental ideal in the Witt ring. Moreover, there is an epimorphism from the fourth unramified cohomology group $\Hnr^{4}(X, \mu_{2}^{\otimes 4})$ to 
$ \Hnis^2 (X, \I^3 (\mathcal L)) $ with finite kernel. 

The sheaves $\piA{2}{\A^2 \setminus 0}$ and $\piA{3}{\A^2 \setminus 0}$ are $\A^1$-homotopy sheaves of the motivic sphere $\A^2 \setminus 0$. 
The computation of classical homotopy groups of spheres is a notoriously difficult problem in topology and the computation of homotopy sheaves of motivic spheres can certainly be regarded as even more difficult. We show that the group 
$\Hnis^{3}(X,\piA{2}{\A^2 \setminus 0}(\mathcal{L}))$ is a quotient of the significantly more computable group $\CH^3 (X)/2$. 
For this, we use a description of the sheaf $\piA{2}{\A^2 \setminus 0}$ due to Asok--Fasel \cite[Theorem 3]{AF-3fold}. However, we perform our cohomological computations with the added difficulty of working in dimension $4$. The bulk of our cohomological computations is dedicated to the group 
$\Hnis^{4}(X,\piA{3}{\A^2 \setminus 0}(\mathcal{L}))$. 
The difficulty of computing this cohomology group stems from the fact that the unstable $\mathbb A^1$-homotopy sheaf $\piA{3}{\A^2 \setminus 0}$ has not even remotely been computed! 
Nonetheless, we prove the following vanishing theorem, which is certainly of independent interest:
\begin{theorem}[see {\Cref{thm:stage3}}]\label{theorem-2} Let $X$ be a smooth affine variety of dimension $4$ over an algebraically closed field $k$ of characteristic not equal to $2$ or $3$. Then $\Hnis^{4}(X,\piA{3}{\A^{2}\setminus 0}(\mathcal{L})) = 0$ for any line bundle $\mathcal{L}$ on $X$.
\end{theorem}
Under the assumptions of \Cref{theorem-1,theorem-2}, the third-named author gave a cohomological criterion for all stably trivial vector bundles of rank $2$ over $X$ to be trivial \cite{Sy1}. 
In particular, the third-named author proved that all stably trivial vector bundles of rank $2$ over $X$ are trivial if $\Hnis^2(X, \textbf{I}^{3}) = \CH^{3} (X) = \CH^{4} (X) = 0$ \cite[Corollary 3.20]{Sy1}. \Cref{theorem-1} strengthens this criterion significantly. 
Furthermore, we remark that the map
\[
(c_{1},c_{2})\: \mathcal{V}_{2}(X) \rightarrow \CH^{1}(X) \times \CH^{2}(X)\]
also fails to be surjective in general. 
The image of the map was analyzed in \cite[Theorem 2.2.2]{AFH-alg19} 
and can be described in terms of motivic Steenrod operations. 

\Cref{theorem-2} actually goes \textit{beyond} the analysis of the fibers of the Chern class map above.
 If $\mathcal{E}$ is an algebraic vector bundle of rank $r \leq 4$ over a smooth affine fourfold 
$X$ as in the theorems above, it is natural to ask for a sufficient cohomological criterion for 
$\mathcal{E}$ to split off a trivial vector bundle of rank $1$. 
If $r=4$, M. P. Murthy's celebrated work shows that $\mathcal{E}$ 
splits off a trivial vector bundle of rank $1$ if and only if $c_{4}(\mathcal{E}) = 0 \in \CH^{4}(X)$ \cite[Theorem 3.7]{Mu}. 
Similarly, if $r=3$, A. Asok and J. Fasel shows that $\mathcal{E}$ splits off a trivial vector bundle of rank $1$ if and only if $c_{3}(\mathcal{E}) = 0 \in \CH^{3}(X)$ \cite[Theorem 2]{AF-splitting}. From an obstruction-theoretic viewpoint, 
the question becomes only harder for rank $2$ bundles, since more 
cohomological obstructions to splitting off a trivial line bundle occur. The vanishing theorem above allows us to prove the following result:

\begin{theorem}[see {\Cref{thm:ch3-zero-determined-euler}}]\label{intro-ch3-zero-euler} Let $X$ be a smooth affine fourfold over an algebraically closed field $k$ of characteristic not equal to $2$ or $3$, and suppose that $\CH^3(X)/2=0$. Let $\mathcal{L}$ be a line bundle over $X$ and let $\Vect^{\mathcal L}_2(X)$ denote the set of isomorphism classes of $\mathcal{L}$-oriented algebraic vector bundles of rank $2$ on $X$ (see \Cref{def:L-or}). Then, by taking the Euler class, we obtain a bijection
\[e_{\mathcal L}\: \Vect^{\mathcal L}_2(X) \xto{\sim} \CHW^2(X, \mathcal L).\]
In particular, an algebraic vector bundle $\mathcal{E}$ of rank $2$ splits off a trivial vector bundle of rank $1$ if and only if the Euler class of $\mathcal E$ is zero in the Chow--Witt group $\CHW^{2}(X, \mathcal L)$.
\end{theorem}

As another considerable consequence of \Cref{theorem-2}, we prove a cancellation theorem for symplectic vector bundles of rank $2$ over smooth affine fourfolds (\Cref{thm:sp-inj}). Finally, if $X$ is a smooth affine variety of dimension $4$ over $\mathbb{C}$, we denote by $\mathcal{V}_{2}^{\top}(X)$ the set of isomorphism classes of complex topological vector bundles of rank $2$ over the complex manifold $X (\mathbb{C})$. \Cref{theorem-2} also allows us to prove some statements about the fibers of the map $\mathcal{V}_{2}(X) \rightarrow \mathcal{V}_{2}^{\top}(X)$ induced by complex realization.

We remark that the classification of complex topological vector bundles via Chern classes has a long history itself (cf. \cite{AR,MilnorStasheff, Switzer2}). More recent work by the second-named author and others provides insight on how to distinguish non-isomorphic bundles with the same Chern classes (cf. \cite{Opie, Hu, CHO24, Opie2, Yang23}). These works rely on an obstruction-theoretic argument, which implies that, over any compact manifold (or finite cell complex), there are only finitely many isomorphism classes of complex topological vector bundles with prescribed Chern classes; a key element in the argument is the fact that the homotopy fiber of the topological universal Chern class map has finite homotopy groups.  Motivated by the analogy between algebraic and topological vector bundles, it is natural to seek cohomological criteria implying finiteness for algebraic vector bundles with fixed Chern classes. Notably, motivic obstruction theory does not seem to yield any general finiteness statement. Failure of such an argument can be related to the fact that there are no known finiteness results for motivic homotopy sheaves of spheres. The additional challenges in understanding to what extent algebraic Chern classes determine the isomorphism class of an algebraic bundle are therefore related to fundamental differences between the topological and algebraic categories.

\subsection{Acknowledgements} We would like to thank Aravind Asok, Patrick Brosnan, Jean Fasel, James Hotchkiss, Michael J. Hopkins, Danny Krashen, Stefan Schreieder, Brian Shin, Burt Totaro and Kirsten Wickelgren for conversations and correspondences around the ideas in this paper. The authors would also like to thank PCMI for the program on motivic homotopy theory, where much of this work was done. The first-named author is supported by NSF DMS-2303242. The second-named author was partially supported by NSF DMS-2202914. The third-named author was partially funded by the Deutsche Forschungsgemeinschaft (DFG, German Research Foundation) - Project numbers 461453992; 544731044.

\subsection{Conventions}In what follows, let $X$ be a smooth separated scheme of finite type over a field $k$ and $\mathcal X$ be a motivic space over $k$. Throughout:
\begin{itemize}
\item $k$ denotes an algebraically closed field.
\item $\Sm_k$ denotes the category of smooth, separated, finite-type $k$-schemes. We write $\mathcal H(k)$ for the homotopy category of motivic spaces and $\mathcal H_{\bullet}(k)$ for the pointed version.
\item $\piA{n}{\mathcal X}$ denotes the $n$-th  $\mathbb A^1$-homotopy sheaf of $\mathcal X$.
\item Given motivic spaces $\mathcal X$ and $\mathcal Y$, we write $[\mathcal X, \mathcal Y]$ for $\mathbb A^1$-homotopy classes of maps from $\mathcal X$ to $\mathcal Y$, i.e., for the set of morphisms from $\mathcal X$ to $\mathcal Y$ in $\mathcal H(k)$. We use an analogous notation for morphisms in $\mathcal H_{\bullet}(k)$ between pointed motivic spaces.
\item By a ``homotopy class", we mean an $\mathbb A^1$-homotopy class.
\item All sheaves are understood to be Nisnevich sheaves unless otherwise stated.
\item Given a Nisnevich sheaf of abelian groups $\textbf{A}$ on $\Sm_k$, we write $\Hnis^\ast( X,\textbf{A})$ for the Nisnevich cohomology of $X$ with coefficients in $\textbf{A}$.  We will also use Zariski (resp. \'etale) cohomology of Zariski (resp. \'etale) sheaves, occasionally. We will always write $\Hzar^\ast$ (resp. $\Het^\ast$) in this case. 
\item We write $\Hmot^{i}(X,\Z(j))$ (resp. $\Hmot^i(X,\Z/2(j))$) for motivic cohomology in degree $i$ of the complex of sheaves $\Z(j)$ (resp. $\Z/2(j)$).
\item $\CH^m( X)$ denotes the $m$-th Chow group of $X$ of codimension $m$ cycles and $\Ch^m(X)$ for $\CH^m(X)/2$.
\end{itemize}

\section{Background and tools}

We assume some familiarity with unstable motivic homotopy theory. For example, we do not introduce the category of {\em motivic spaces} over a field $k$, by which we mean the category of $\A^1$-invariant Nisnevich sheaves (of spaces) on the category of smooth $k$-schemes \cite{MV99}. Most of the computations in this paper take place in the homotopy category of motivic spaces over $k$ and involve analysis of $\A^1$-homotopy classes of maps between specific motivic spaces.

The reader should also be familiar with {\em affine representability} for algebraic vector bundles (\cite[\S 8.1]{Morel} and \cite[Theorem 1]{AHW1}), which gives a bijection between rank $r$ algebraic vector bundles on a smooth affine variety over $k$ and the set of $\mathbb A^1$-homotopy classes of maps from $X$ to the classifying space $\BGL_r$. 
Some knowledge of Moore--Postnikov towers in classical or motivic homotopy theory is also recommended (cf. \cite[Appendix~B]{Morel} and \cite[\S 6.1]{AF-splitting}), although we summarize some aspects of the theory in \Cref{subsec:def-twists}. 

Below, we briefly summarize key concepts as needed for this paper, using this as an opportunity to clarify notation and establish conventions. We discuss properties of strictly $\A^1$-invariant sheaves (\Cref{subsec:strictly}), contractions and twists of sheaves (\Cref{subsec:contraction} and \Cref{subsec:actions}), the Rost--Schmid complex (\Cref{subsec:rost-schmidt}), and motivic Moore--Postnikov theory (\Cref{subsec:def-twists}). 
We also summarize some useful cohomological facts in \Cref{subsec:useful}.

\subsection{Strictly $\mathbb A^1$-invariant sheaves of abelian groups}\label{subsec:strictly} If $\textbf{A}$ is a (pre)sheaf of abelian groups on $\Sm_k$, we say that it is $\A^1$\emph{-invariant} if the projection $X \times \A^1 \to X$ induces an isomorphism $\textbf{A}(X) \to \textbf{A}(X \times \A^1)$ for every $X\in \Sm_k$. If $\textbf{A}$ is a Nisnevich sheaf of abelian groups on $\Sm_k$, we say it is \textit{strictly $\A^1$-invariant} if $\Hnis^{n}(-,\textbf{A})$ is $\A^1$-invariant for each $n\ge 0$. A prototypical example of a strictly $\A^1$-invariant sheaf is the homotopy sheaf $\piA{n}{\mathcal{X}}$ associated to a motivic space $\mathcal{X}$ when $n\ge 2$.\footnote{Historically, distinctions were made between {\em strongly} and {\em strictly} $\mathbb A^1$-invariant sheaves. Due to work of Morel, strongly $\mathbb A^1$-invariant Nisnevich sheaves of abelian groups on $\Sm_k$ are actually also strictly $\mathbb A^1$-invariant. See also \cite{bachmannstrictlystrongly}.} 
Strictly $\A^1$-invariant Nisnevich sheaves of abelian groups assemble to form an abelian category \cite[6.24]{Morel}.

Key examples of strictly $\A^1$-invariant sheaves include \textit{Milnor $K$-theory}, denoted $\KM_n$, \textit{Milnor--Witt $K$-theory}, denoted $\KMW_n$, and the \textit{fundamental ideal} $\I$ as well as its powers $\I^n$. We have canonical short exact sequences
\begin{equation}\label[sequence]{eqn:M-MW-I-sequence}
    0 \to \I^{n+1}\to \KMW_n \to \KM_n \to 0
\end{equation}
and
\begin{equation}\label[sequence]{eqn:I-I-M}
0\to \I^{n+1}\to \I^n \to \KM_n/2 \to 0.
\end{equation}

Other important examples of strictly $\A^1$-invariant sheaves are \textit{higher Grothendieck Witt sheaves} $\GW_i^j$, which are the Nisnevich sheafifications of \textit{higher Grothendieck Witt groups} $GW^i_j (X)$ (cf. \cite{Schlichting-a,Schlichting-b,Schlichting17}; we refer the reader to \cite[Section 4]{AF-3fold} for a concise treatment). For any $i,j \in \mathbb{Z}$, $GW^i_j (X)$ are $4$-periodic in $j$.
For $i \geq 0$ and $X$ affine, the groups $GW_i^j (X)$ can be identified with Hermitian $K$-theory groups as defined by M. Karoubi \cite{Karoubi73}. In particular:

\begin{proposition}[{\cite[Corollary A.2]{Schlichting17}}]\label{GW-KSp}
For $i \geq 0$, let $\KO_i$ and $\KSp_{i}$ denote the Nisnevich sheafifications of the $i$-th orthogonal and symplectic $K$-theory groups, respectively. There are isomorphisms
    \begin{align*}
        \KO_i &\cong \GW_i^0\\
        \KSp_i &\cong \GW_i^2
    \end{align*}
of strictly $\A^1$-invariant sheaves.
\end{proposition}

\begin{rmk} Grothendieck--Witt groups have many interesting properties and have been extensively studied. For example, for $i > 0$ the negative Grothendieck--Witt groups coincide with triangular Witt groups defined by P. Balmer \cite{Balmer} via the formula $GW_{-i}^{j} (X) = W^{i+j}(X)$. By \cite[\S 3]{Hornbostel}, the Grothendieck-Witt sheaves $\GW_i^j$ are the homotopy sheaves $\piA{i}{\mathcal{GW}^j}=\GW_i^j$ of a motivic space $\mathcal{GW}^j$.\end{rmk}

\subsection{Contractions of strictly $\A^1$-invariant sheaves}\label{subsec:contraction}

We will be concerned with a few elementary constructions on sheaves and presheaves, which we briefly recall here.
\begin{const}[Contraction of presheaves]\label{def:contraction} Let $\textbf{A}$ be a sheaf of abelian groups or pointed sets on $\Sm_k$. The {\em contraction} of $\textbf{A}$, denoted $\textbf{A}_{-1}$, is the sheaf on $\Sm_k$ defined by \[\textbf{A}_{-1}(U)=\op{ker}\left(\textbf{A}(\mathbb G_m \times U) \xto{(1\times \op{id})^*} \textbf{A}(U)\right) .\] 

We define the {\em$i$-fold contraction } $\textbf{A}_{-i}$ inductively for $i \in \Z_{\geq 2}$ by 
\[\textbf{A}_{-i}=(\textbf{A}_{-(i-1)})_{-1}.\]
\end{const}
By \cite[Lemmas 5.32 and 7.33]{Morel}, $\textbf{A} \mapsto \textbf{A}_{-i}$ is functorial on strictly $\mathbb A^1$-invariant sheaves and preserves exact sequences. We record the following simple descriptions of contractions of some important strictly $\A^1$-invariant sheaves: 

\begin{proposition}\label{prop:contractions-1} 
For any $i$ and $j$, we have the following.
\begin{enumerate}
\item\label[empty]{KM-contractions}  $(\KM_i)_{-1}\cong \begin{cases} \KM_{i-1} & i \geq 1\\ 0& i=0.\end{cases}$
\item\label[empty]{KMW-I-contractions} $(\I^n)_{-1}\cong \I^{n-1}$ and $(\KMW_i)_{-1}\cong \KMW_{i-1}$.
    \item\label[empty]{GW-contractions}
    $
        \left( \GW_i^j \right)_{-1}\cong \GW_{i-1}^{j-1}.$
    \item\label[empty]{GW-vanishing}
    If $j>i+1$ then $  \left( \GW_{i+1}^i \right)_{-j} = 0.$
\end{enumerate}
These results are standard, but can be found for instance in \cite[Lemma 2.7, Proposition 2.9]{asok2014algebraic}, \cite[Proposition 4.4]{AF-3fold}, and \cite[Proposition 3.4.3]{AF-splitting}, respectively.
\end{proposition}

\subsection{Actions and twisted cohomology}\label{subsec:actions}

Given a sheaf of abelian groups $\textbf{A}$, a sheaf of groups $\mathbf{G}$ acting on $\textbf{A}$,  a smooth scheme $X \in \Sm_k$, and a Nisnevich $\mathbf{G}$-torsor $\mathcal P$ over $X$, cohomology of $X$ with coefficients in $\textbf{A}$ can be twisted by $\mathcal P$ as follows.

\begin{construction}[Twisted cohomology of Nisnevich sheaves]\label{const:twists} With set-up as above the $\mathcal P$-twisted cohomology of $X$ with coefficients in $\textbf{A}$ is defined as the cohomology of the sheaf $\Z[\mathcal P] \otimes_{\Z[\mathbf{G}]}\textbf{A}$ on the small Nisnevich site of $X$.
\end{construction}

We summarize how to compute twisted cohomology groups in \Cref{subsec:rost-schmidt}. Below, we give examples of sheaves of abelian groups with group actions that we will use extensively.
\begin{example}\label{ex:contract-act}
Contracted sheaves carry a natural action of $\mathbb G_m$ (see, e.g., \cite[\S 2.4]{AF-splitting}). The action of $\mathbb G_m$ arises from an action of $(\KMW_0)^{\times}$ on contracted sheaves, which is then restricted along the natural map $\mathbb G_m \to \KMW_0$ given by $a \mapsto 1 + \eta [a]$ over any field $F$ over $k$ and $a \in F^{\times}$. In particular, $\mathbb G_m$ acts on $\KMW_i$ for each $i$; this $\mathbb G_m$-action coincides with the $\mathbb G_m$-action coming from the product morphism $\KMW_0 \times \KMW_i \to \KMW_i$.
\end{example}

\begin{example}\label{ex:I-act} From \Cref{ex:contract-act}, we get a $\mathbb G_m$-action on $\I^j$. This can be seen either from the fact that it is an ideal in $\KMW_j$ and therefore closed under the ($\KMW_0)^\times$-action or from \Cref{prop:contractions-1}\Cref{KMW-I-contractions}. Unraveling definitions, one checks these actions agree. \end{example}

\begin{example}\label{ex:KM-act} By \Cref{eqn:M-MW-I-sequence} or \Cref{prop:contractions-1}\Cref{KM-contractions}, we get a $\mathbb G_m$-action on $\KM_i$ for each $i$ induced by the action of $\KMW_i$; this action is trivial. \end{example}

\begin{rmk}\label{rmk:act-compatibility} By definition, the sequences \Cref{eqn:M-MW-I-sequence} and \Cref{eqn:I-I-M} are compatible with $\mathbb G_m$-actions and induce short exact sequences of twisted sheaves and long exact sequences on twisted cohomology. \end{rmk}

\begin{example}\label{ex:c1xc2-act} Given any fibration in motivic spaces with $\A^1$-simply connected fiber, then the $\A^1$-fundamental group of the base acts on the $\A^1$-homotopy sheaves of the fiber. For example, we will later consider the fibration
\[\mathcal F \to \BGL_2 \xrightarrow{c_1\times c_2} K(\KM_1,1) \times K(\KM_2,2),\]
coming from \Cref{eq:c1c2}.
This induces an action of $\piA{1}{\piA{1}{K(\KM_1,1) \times K(\KM_2,2)}} \cong\piA{1}{\BGL_2} \cong \mathbb G_m$ on $\piA{2}{\mathcal F} \cong \I^3$. By \cite[Proposition 6.3]{AF-3fold}, this agrees with the action on $\I^3$ discussed in \Cref{ex:I-act} above.
\end{example}

\subsection{The Rost--Schmid complex}\label{subsec:rost-schmidt}

We briefly recall Rost--Schmid complexes, which are the main computational tool for computing (twisted) Nisnevich cohomology of strictly $\A^1$-invariant sheaves. Our main reference is \cite[\S 5]{Morel} and we refer the reader to \cite[Section 2]{AF-3fold} or \cite[Section 6]{bachmannstrictlystrongly} for a concise discussion.

A strictly $\A^1$-invariant sheaf $\textbf{A}$ of abelian groups admits an associated Rost--Schmid complex $C^{\ast}_{RS} (X, \textbf{A})$ computing Nisnevich cohomology of a smooth $k$-scheme $X$ with coefficients in $\textbf{A}$. The terms of this complex take the form
\[C^i_{RS}(X,\textbf{A})=\bigoplus_{\substack{x \in X^{(i)}}} \textbf{A}_{-i}(x,\omega_{x/X})\]
where $X^{(i)}$ denotes the set of codimension $i$ points of $X$ \cite[Definition 5.7]{Morel}. The cohomology of the Rost--Schmid complex $C^{\ast}_{RS} (X, \textbf{A})$ coincides with the Nisnevich (and Zariski) cohomology groups of the sheaf $\textbf{A}$ in each degree \cite[Corollary 5.43]{Morel}.

Now let $X \in \Sm_k$ and $\mathcal{L}$ be a line bundle over $X$. Then any contraction $\textbf{A}_{-1}$ of a strictly $\A^1$-invariant sheaf of abelian groups $\textbf{A}$ has a canonical $\mathbb{G}_{m}$-action by Example \ref{ex:contract-act}. This allows to define the twisted sheaf $\textbf{A}_{-1}(\mathcal{L})$ over the small Nisnevich site over $X$. One can then analogously define the twisted Rost--Schmid complex $C^{\ast}_{RS}(X, \mathcal{L}; \textbf{A}_{-1})$ with terms of the form
\[\bigoplus_{\substack{x \in X^{(i)}}} \textbf{A}_{-(i+1)}(x,\omega_{x/X} \otimes \mathcal{L})\]
in degree $i$. The cohomology groups of the complex $C^{\ast}_{RS}(X, \mathcal{L}; \textbf{A}_{-1})$ compute the Nisnevich cohomology groups of the sheaf $\textbf{A}_{-1}(\mathcal{L})$ in a given degree \cite[Remark 5.14.2]{Morel}. For completeness, we note some easy consequences of the definition of the Rost--Schmid complex that will be useful to us throughout what follows:
\begin{lemma}\label{prop:gersten-complex-properties} \textit{(Elementary properties of Rost--Schmid complexes)}
\begin{enumerate}
    \item\label[empty]{contract-isom} Suppose that $\mathcal{F}\to \mathcal{G}$ is a map of strictly $\A^1$-invariant Nisnevich sheaves of abelian groups that induces an isomorphism after $n$-fold contraction. For any $X$ and any line bundle $\mathcal L$ over $X$, the induced map
\[ \Hnis^i\left(X,\mathcal{F}(\mathcal L)\right) \to \Hnis^i\left(X,\mathcal{G}(\mathcal L)\right)\] is an isomorphism for $i\ge n+1$.
    \item\label[empty]{sheaf-contracts-to-zero}
    If $(\mathcal{F})_{-n}$ restricts to zero on the small Nisnevich site over $X$, then $
        \Hnis^i(X,\mathcal{F}) = 0$
    for $i\geq n$.
\end{enumerate}
\end{lemma}

\subsection{Motivic Moore--Postnikov theory}\label{subsec:def-twists}

We give a brief summary of twisted Eilenberg--Mac Lane spaces and Moore--Postnikov theory, following \cite[\S 6.1]{AF-splitting} to which we refer the reader for more details. 

Let $\textbf{G}$ be a sheaf of groups and $\textbf{A}$ a sheaf of abelian groups with a $\textbf{G}$-action. Then $\textbf{G}$ also acts on $K(\textbf{A},n)$ for each $n$ and we define\[ K^{\textbf{G}}(\textbf{A},n)=EG \times^{\textbf{G}} K(\textbf{A},n).\] The space $K^G(\textbf{A},n)$ admits a split morphism to $\mathrm{B}\textbf{G}$ induced by projection onto the first factor.

Consider a morphism $f\:\mathcal E \to \mathcal B$ of pointed motivic spaces, where its homotopy fiber $\mathcal F:=\fib(f)$ is $\A^1$-simply connected and both $\mathcal E$ and $\mathcal B$ are $\A^1$-connected. The $\mathbb A^1$-homotopy sheaves of $\mathcal F$ then carry a natural action of $\piA{1}{\mathcal B}$. By \cite[Theorem 6.1.1]{AF-splitting}, the Moore--Postnikov tower for $\mathcal E \xto{f} \mathcal B$ takes the form

\begin{equation}\label[diagram]{diag:post_general} \begin{tikzcd}[column sep=4em, row sep=2em]
& & \vdots \\
  & &  \mathcal E^{(4)} \ar[d]\\
 &&\mathcal E^{(3)}\ar[d] \ar[r,"k_4"]& K^{\piA{1}{\mathcal B}}(\piA{4}{\mathcal F},5) \\
 && \mathcal E^{(2)}\ar[d] \ar[r,"k_3"]& K^{\piA{1}{\mathcal B}}(\piA{3}{\mathcal F},4) \\
   \mathcal E \ar[uuurr,"\,\,g^{(4)}" {near end, below}]\ar[rr,"g^{(1)}" {near end, below}]\ar[urr,"g^{(2)}"{near end, below}]\ar[uurr,"\,g^{(3)}" {near end, below}] && \mathcal E^{(1)}= \mathcal B \ar[r,"k_2"]& K^{\piA{1}{\mathcal B}}(\piA{2}{\mathcal F},3).
\end{tikzcd} \end{equation}
The indicated morphisms $\mathcal E \to \mathcal E^{(i)}$ induce an equivalence between $\mathcal E$ and the homotopy limit of the system $\{g^{(i)}\}$.  Thus, given a smooth $k$-scheme $X \in \Sm_k$, a morphism $f\: X \to \mathcal E$ is equivalent to a compatible system of morphisms $f^{(i)}\: X \to \mathcal E^{(i)}$.

Given $f^{(i)}\: X \to \mathcal E^{(i)}$, let $\mathcal P$ denote the $\piA{1}{\mathcal B}$-torsor over $X$ associated to the composite 
\[X \xrightarrow{f^{(i)}} \mathcal E^{(i)} \xrightarrow{k_{i+1}} K^{\piA{1}{\mathcal B}}({\piA{i+1}{\mathcal F}},i+2) \to B{\piA{1}{\mathcal B}}.\] A lift of $g^{(i)}$ to $\mathcal E^{(i+1)}$ exists if and only if an element associated to $k_i\circ g^{(i)}$ in the group
\begin{equation}\label[empty]{eq:obst1} \Hnis^{i+2}(X,\piA{i+1}{\mathcal F}(\mathcal P)).\end{equation}
is zero. Moreover, the group
\begin{equation}\label[empty]{eq:choice1} \Hnis^{i+1}(X,\piA{i+1}{\mathcal F}(\mathcal P))\end{equation} acts transitively on choices of lifts of $g^{(i)}$ to $\mathcal{E}^{(i+1)}$. The formalism above implies that for $X$ a smooth affine variety of Nisnevich cohomological dimension at most $d$, there is a bijection
\[[X,\mathcal E] \cong[X,\mathcal E^{(d)}].\]
\begin{rmk}[The special case of Postnikov towers]\label{rmk:postnikov} To recover the Postnikov tower of motivic space $\mathcal X$, we consider the Moore--Postnikov tower of the morphism $\mathcal X \to *$. See \cite[Section 6]{AF-3fold} to see discussion of the formalism in this case. 
\end{rmk}
\begin{rmk}[Comparing Moore--Postnikov towers]\label{rmk:canonicity} Consider a homotopy commutative square of motivic spaces

\begin{equation}\label[diagram]{tower-comparison}\begin{tikzcd} \mathcal E'\ar[d,"e"] \ar[r,"g"]& \mathcal B'\ar[d,"b"] \\
\mathcal E \ar[r,"f"] &\mathcal B,
\end{tikzcd}\end{equation}
where all spaces indicated are $\mathbb A^1$-connected and the vertical homotopy fibers are $\mathbb A^1$-simply connected. Let $\mathcal F'$ denote the homotopy fiber of $g$ and $\mathcal F$ that of $f$. 

Given a Moore--Postnikov tower for $g$ and for $f$, for each $i$ we have comparison maps 
\[e^{(i)}\:{ \mathcal E'}^{(i)}\to \mathcal E^{(i)}\] and 
\[r^{(i)}\:K^{\piA{1}{\mathcal B'}}(\piA{i+1}{\mathcal F'},i+2) 
\to K^{\piA{1}{\mathcal B}}(\piA{i+1}{\mathcal F},i+2)\] fitting into a homotopy commutative diagram
\begin{equation}\label[diagram]{eq:ri}
\begin{tikzcd}
\mathcal E'^{(i)} \ar[r,"{k}'_{i+1}"] \ar[d,"e^{(i)}"]&K^{\piA{1}{\mathcal B'}}(\piA{i+1}{\mathcal F'},i+2) \ar[d,"r^{(i)}"]\\
\mathcal E^{(i)} \ar[r,"{k}_{i+1}"] &
K^{\piA{1}{\mathcal B}}(\piA{i+1}{\mathcal F'},i+2)
\end{tikzcd}
\end{equation}
and making various other diagrams involving the maps $f,g,e,b$ homotopy commutative. 

Let $\mathcal X$ be a motivic space. Given a homotopy commutative diagram
\[\begin{tikzcd}
\mathcal X \ar[r,"h^{\prime(i)}"]\ar[dr,"h^{(i)}" below left]& \mathcal E'^{(i)}\ar[d,"e^{(i)}"] \\
& \mathcal E^{(i)},
\end{tikzcd}
\]
the map on cohomology induced by $r^{(i)}$ from \Cref{eq:ri} induces a map on cohomology with the property that the image of the obstruction to lifting (resp., choices of lift of) $h^{\prime(i)}$ to $\mathcal{E}^{\prime(i+1)}$ maps to the the obstruction to lifting (resp., choices of lift of) $h^{(i)}$ to $\mathcal{E}^{(i+1)}$. 
\end{rmk}

We will be interested in understanding \Cref{diag:post_general} above when $\mathcal E=\BGL_2$, $\mathcal B=K(\KM_1,1)\times K(\KM_2,2)$, and $f=(c_{1}, c_{2}): \BGL_2 \to K(\KM_1,1)\times K(\KM_2,2)$ is the product of universal first and second Chern class morphisms. We will discuss this in detail in Section 3.

\subsection{Vanishing and divisibility results and an exact sequence}\label{subsec:useful}

The sheaves $\I^j$, $\KM_n$, $\KMW_n$, and $\GW_i^j$ introduced in \Cref{subsec:strictly} will appear throughout our paper. We collect some facts about these sheaves that will be repeatedly used in what follows.

\begin{proposition}[{\cite[5.1,~5.2]{AF-3fold}}]\label{prop:Ij-vanishing} 
 Let $X$ be a smooth scheme of dimension $d\ge2$ over an algebraically closed field $k$, and let $\mathcal L$ be a line bundle on $X$.
\begin{enumerate}
    \item\label[empty]{a-Ij-vanishing} The restriction of the sheaf $\I^j(\mathcal{L})$ to the small Nisnevich site over $X$ is identically zero for $j \ge d+1$.
    \item Both $\Hnis^d(X,\I^j(\mathcal{L}))$ and $\Hnis^{d-1}(X,\I^j(\mathcal{L}))$ vanish for $j\ge d$.
\end{enumerate}
\end{proposition}

We will later use the following divisibility results for cohomology with coefficients in $\KMW_d$ and for Chow groups:

\begin{lemma}[{\cite[Lemmas 4.0.3 and 4.0.5]{fasel2021suslins}}]\label{lem:fasel-KMW-divisibility}  Let $X$ be a smooth affine scheme of dimension $d\geq 4$ over an algebraically closed field $k$. Then $\Hnis^{d-1}(X, \KMW_d)$ is uniquely divisible prime to the characteristic of $k$ and $\Hnis^{d-2}(X,\KMW_d)$ is divisible prime to the characteristic of $k$.
\end{lemma}

\begin{theorem}[{\cite{Srinivas}}]
\label{thm:chow-divisibility}
Let $X$ be a smooth affine domain of dimension $d\ge 3$ over an algebraically closed field. Then $\CH^d(X)$ is divisible and torsion-free.
\end{theorem}

The following exact sequence relates the cohomology of some Grothendieck-Witt sheaves to algebraic cycles:

\begin{proposition}[{\cite[Proposition 4.16]{AF-3fold}}]\label{GW-SES} 
 For $X$ a smooth scheme over $k$ with $2 \in k^{\times}$, there is an exact sequence
  \[
        \Ch^{2}(X) \xrightarrow{\Sq^2_\mathcal{L}} \Ch^3(X) \to \Hnis^3(X, \GW_3^{2}(\mathcal{L})) \to 0.
  \]
\end{proposition}
Here the homomorphism $\Sq^2_{\mathcal{L}}$ is the twisted Steenrod operation; we refer the reader to \cite[Theorem 4.17]{AF-3fold} for more details.

\section{The Moore--Postnikov tower for a Chern class map}\label{sec:Moore-Post-Chern}

Let $X$ be a smooth affine fourfold over an algebraically closed field $k$ of characteristic not equal to $2$ or $3$. We analyze the set of isomorphism classes of algebraic vector bundles of rank $2$ over $X$ with fixed Chern data by use of an appropriate Moore--Postnikov tower. The first consequence of the this will be a proof of \Cref{theorem-1}. 

Following \cite[2.1.2]{AFH-alg19}, for each $i$ we consider the first non-trivial Postnikov section in the Postnikov tower of $K(\Z(i),2i)$ (for a brief discussion of Postnikov towers in $\A^1$-homotopy theory, see  \Cref{rmk:postnikov} above): this is a morphism $K(\Z(i),2i)\to K(\KM_i,i)$. Precomposing this map with the universal $i$-th Chern class morphism $\BGL_i \to K(\Z(i),2i)$, we obtain maps \[c_i\colon \BGL_i \to K(\KM_i,i)\] representing the Chow-valued $i$-th Chern class in the unstable motivic homotopy category $\mathcal H(k)$. Consider the product of the first two Chern classes
\begin{equation}\label[diagram]{eq:c1c2}
    (c_{1},c_{2})\: \BGL_2 \to K(\KM_1,1) \times K(\KM_2,2).
\end{equation}
We study the Moore--Postnikov tower of \Cref{eq:c1c2}, referring the reader to \Cref{subsec:def-twists} for a brief overview of motivic Moore--Postnikov theory and to \cite[\S 6.1]{AF-splitting} for details. Since $\piA{1}{\BGL_2}\cong \mathbb G_m$, this is a tower of principal fibrations over twisted Eilenberg--MacLane spaces of the form $K^{\mathbb G_m}(\piA{i+1}{\mathcal F},i+2)$, where $\mathcal F$ is the $\A^1$-homotopy fiber of $(c_{1},c_{2})$. By \cite[Proposition 2.2.1]{AFH-alg19}, we obtain a tower of twisted principal fibrations:
\begin{equation}\label[diagram]{diag:post_GL} \begin{tikzcd}[column sep=small]
  & &  \BGL_2^{(4)} \ar[d]\\
 && \BGL_2^{(3)}\ar[d] \ar[r,"k_4"]& K^{\mathbb G_m}(\piA{4}{\BGL_2},5) \\
 && \BGL_2^{(2)}\ar[d] \ar[r,"k_3"]& K^{\mathbb G_m}(\piA{3}{\BGL_2},4) \\
   \BGL_2 \ar[rr]\ar[urr]\ar[uurr] && K(\KM_1,1) \times K(\KM_2,2) \ar[r,"k_2"]& K^{\mathbb G_m}(\I^3,3).\\
\end{tikzcd} \end{equation}
In the above, $\I^3$ is the third power of the fundamental ideal sheaf, fitting into a short exact sequence of sheaves $0\to \I^3 \to \KMW_2 \to \KM_2\to 0$, as noted in \Cref{eqn:M-MW-I-sequence}. For $X$ a smooth affine variety of Nisnevich cohomological dimension at most $4$,
$$[X,\BGL_2] \cong [X,\BGL_2^{(4)}].$$
Thus \Cref{diag:post_GL} shows that we may build a lift of a morphism $X \to K(\KM_1,1) \times K(\KM_2,2)$ to a morphism $X \to \BGL_2$ in three stages. The obstructions and choices of lifts at each stage in the tower are controlled by Nisnevich cohomology with twisted coefficient sheaves $\mathbf{I}^3$ or $\piA{i}{\BGL_2}$ for $i=3$ or $i=4$, as appropriate. By working with these twisted cohomology groups we will obtain the following:
\begin{enumerate}
\item[(Stage 1)] We study lifts of a given morphism $X\to K(\KM_1,1)\times K(\KM_2,2)$ to $\BGL_2^{(2)}$. We show that lifts up to homotopy are a torsor for a quotient of $\Hmot^5(X,\Z/2(3))$. We obtain conditions under which the set of lifts is finite (resp. a singleton) in \Cref{sec:stage1}.
\item[(Stage 2)] We study the set of lifts of a given morphism $X\to BGL_2^{(2)}$ to $\BGL_2^{(3)}$.  Lifts up to homotopy are a torsor for a group $\Hnis^3(X, \piA{3}{\BGL_2}(\mathcal{L}))$, where $\mathcal L$ is a line bundle on $X$. We show that $ \Hnis^3(X, \piA{3}{\BGL_2}(\mathcal{L}))$ is a quotient of $\Ch^3(X)$ in \Cref{sec:stage2}.
\item[(Stage 3)] We study the set of lifts of a given morphism $X\to  BGL_2^{(3)}$ to $ \BGL_2^{(4)}$. Lifts up to homotopy are a torsor for a quotient of $\Hnis^4(X, \piA{4}{\BGL_2}(\mathcal{L}))$, where $\mathcal L$ is a line bundle on $X$. We prove that $\Hnis^4(X, \piA{4}{\BGL_2}(\mathcal{L}))$ is zero in \Cref{sec:stage3}.
\end{enumerate}
Assuming the stage-by-stage analysis outlined above, we can deduce our first theorem:
\begin{proof}[Proof of \Cref{theorem-1}] If $\Hmot^5(X,\Z/2(3))$ and $\Ch^3(X)$ are trivial (resp. finite), then any given pair in $(a,b)\in \CH^1(X) \times \CH^2(X)$ represent the first and second Chern classes of at most one (resp. at most finitely many) rank $2$ algebraic vector bundles on $X$, since trivial (resp. finite) groups bound the choices of lift at each stage in the Moore--Postnikov tower that parameterizes lifts of $(a,b)\: X \to K(\KM_1,1) \times K(\KM_2,2)$ to a map $X \to \BGL_2^{(4)}$ representing a rank $2$ vector bundle on $X$.
\end{proof}

\subsection{The first Moore--Postnikov stage}\label{sec:stage1}
We now analyze finiteness of lifts
 \begin{equation}\label[diagram]{diag:lift2}\begin{tikzcd}& \BGL_2^{(2)}\ar[d]\\
X \ar[ur,dashed]\ar[r,"{(c_{1}, c_{2})}" below] & K(\KM_1,1) \times K(\KM_2,2)
\end{tikzcd} \end{equation} up to homotopy. Appealing to \Cref{diag:post_GL} and \Cref{subsec:def-twists}, it is sufficient to study finiteness of $\Hnis^2(X,\I^3(\mathcal L))$. If the latter group is finite (resp. trivial), then there are only finitely many lifts (resp. a unique lift). Our main result in this section is the following:
\begin{proposition}\label{lem:stage1} 
Let $X$ be a smooth affine fourfold over an algebraically closed field $k$ with $2$ invertible in $k$. Let $\mathcal L$ be a line bundle on $X$.
\begin{enumerate}
\item[]
\begin{enumerate}
    \item\label[empty]{item:H2KM3} $\Hnis^2(X,\KM_3/2)\cong \Hnis^2(X,\I^3(\mathcal L))$.
    \item\label[empty]{item:mot53} $\Hmot^{5}(X,\Z/2(3)) \cong \Hnis^2(X,\I^3(\mathcal L))$. 
    \item\label[empty]{item:unr} There is an epimorphism $\Hnr^4(X,\mu_2^{\otimes 4}) \to \Hnis^2(X,\I^3(\mathcal L))$ with finite kernel.
\end{enumerate}
\end{enumerate}
\end{proposition}
\begin{remark}\label{rmk:finiteness-stage1} The group $\Hnr^4(X,\mu_2^{\otimes 4})$ is the fourth unramified cohomology of $X$ with coefficients in $\mu_2^{\otimes 4}$ and will be defined below in Definition \ref{def:unramified-cohomology}. The previous proposition implies that if any one of the groups $\Hnis^2(X,\KM_3/2),$ $\Hmot^{5}(X,\Z/2(3))$, $\Hnr^4(X,\mu_2^{\otimes4})$, or $\Hnis^2(X,\I^3(\mathcal L))$ is finite, then all the others are also finite. 
\end{remark} 
We first make some reductions.

\begin{lemma}\label{lem:01}  Let $X$ be a smooth affine fourfold over an algebraically closed field $k$ with $2$ invertible in $k$. Then $\Hnis^2(X,\KM_3/2) \cong \Hmot^{5}(X,\Z/2(3))$. 
\end{lemma}
\begin{proof} By \cite[Theorem 1.3]{totaro03}, we have an exact sequence
\begin{equation}\label[sequence]{seq:tot} H_{\mot}^{i+j}(X,\Z/2(j-1)) \to H_{\mot}^{i+j}(X,\Z/2(j)) \to H_{\Zar}^i(X,\mathcal H_{\et}^j(\mathbb Z/2))  \to  H_{\mot}^{i+j+1}(X,\Z/2(j-1)).\end{equation}
We note a few facts:
\begin{itemize}
\item  By the Milnor conjecture \cite{Voevodsky-mod2}, $\mathcal H_{\et}^j(\Z/2) \simeq \KM_j/2$. 
\item By \cite[Theorem 19.3]{MVW06}, $\Hmot^{n}(X,\Z/2(m))=0$ whenever $n>2m$.
\item By \cite[Theorem 13.10]{MVW06},  $\Hnis^2(X,\KM_3/2)\cong \Hzar^2(X,\KM_3/2)$.
\end{itemize}
Using these facts and taking $i=2,j=3$ in \Cref{seq:tot} yields an exact sequence:
\[0 \to \Hmot^{5}(X,\Z/2(3)) \to \Hnis^2(X,\KM_3/2)  \to  0\]
which completes the proof.
\end{proof}

By \Cref{lem:01}, to prove \Cref{lem:stage1}\Cref{item:H2KM3}, \Cref{item:mot53} it suffices to prove only \Cref{item:H2KM3}. We reduce \Cref{lem:stage1}\Cref{item:H2KM3} as follows.

\begin{lemma}\label{prop:reduce-KMI} Let $X$ be a smooth affine fourfold over an algebraically closed field $k$. Then, if $\Hnis^2(X,\KM_4/2)=\Hnis^3(X,\KM_4/2)=0$, one has $\Hnis^2(X,\KM_3/2)\cong \Hnis^2(X,\I^3(\mathcal L))$.
\end{lemma}
\begin{proof} The sheaf $\I^5|_{X}(\mathcal L)$ is zero by \Cref{prop:Ij-vanishing} so, by \Cref{eqn:I-I-M}, we obtain an isomorphism of sheaves on the small Nisnevich site of $X$
\[ \I^4 \cong \I^4/\I^5\cong \KM_4/2.\] This sequence is compatible with twists. Now consider the short exact sequence of sheaves $
 0\to   \I^4|_{X}(\mathcal L) \to \I^3|_{X}(\mathcal L) \to \KM_3/2|_X\to 0.$
The associated long exact sequence on cohomology and the fact that $\KM_4/2(\mathcal L) \simeq \KM_4/2$ imply the result.
\end{proof}
So, to prove  \Cref{lem:stage1} \Cref{item:H2KM3} it suffices to argue that $\Hnis^2(X,\KM_4/2)$ and $\Hnis^3(X,\KM_4/2)$ are trivial. This follows from a spectral sequence argument that we set up in \Cref{subsec:bloch-ogus}. We prove in \Cref{subsec:c-proof} that $\Hnis^2(X,\KM_3/2)$ is finite if and only if $H_\nr^4(X,\mu_2^{\otimes4})$ is finite.

\subsubsection{Bloch--Ogus sheaves and vanishing of $\Hnis^2(X,\KM_4/2)$}\label{subsec:bloch-ogus}
Given a suitable cohomology theory on a variety, the \textit{coniveau} filtration, dating back to work of Grothendieck, filters cohomology classes by the codimension of their support. This filtration gives rise to the \textit{Bloch--Ogus spectral sequence} (see \cite[\S 1]{colliotetc}). 
The case of interest to us computes \'etale cohomology, and can be written as
\begin{equation}\label[sseq]{eqn:coniveau-SS-E1}
\begin{aligned}
    E_1^{p,q} = \bigoplus_{x\in X^{(p)}}  \Het^{q-p}(k(x), \mu_2^{\otimes j - p}) \Rightarrow \Het^{p+q}(X, \mu_2^{\otimes j}),
\end{aligned}
\end{equation}
where $X^{(p)}$ denotes codimension $p$ points of $X$.\footnote{%
Historically, the Bloch--Ogus spectral sequence is set-up in the Zariski site. In our situation, we can use Nisnevich cohomology instead (see \cite[7.5.3]{colliotetc} for a precise statement).
}
By \cite[Corollary 6.1]{Bloch-Ogus}, the cohomology $E_2$-page of \Cref{eqn:coniveau-SS-E1} takes the form
\begin{equation}\label[sseq]{E2-page}
    \Hnis^p(X,\sHet^q(j)) \implies \Het^{p+q}(X,\mu_2^{\otimes j}),
\end{equation}
where we define $\sHet^q(j)$ to be the Nisnevich sheafification of the presheaf
\[
    U \mapsto \Het^q(U, \mu_2^{\otimes j}).
\]
As a consequence of the Milnor conjectures, certain Bloch--Ogus sheaves can be identified with Milnor $K$-theory modulo $2$.
\begin{theorem}[\cite{Voevodsky-mod2}]\label{thm:milnor-conj-nr} 
For $i,j \geq 0$, let $\mathcal{H}^i_{\et}(j)$ denote the Nisnevich sheafification of the presheaf 
$U \mapsto \Het^i(U,\Z/2(j))$ on $\Sm_k$. Over a base field $k$, there is an isomorphism $\KM_i/2 \cong \mathcal{H}^i_{\et}(i)$ of strictly $\A^1$-invariant Nisnevich sheaves on $\Sm_k$.
\end{theorem}
In our setting, each term on the $E_\infty$-page of the Bloch--Ogus spectral sequence is finite: they are subquotients of \'etale cohomology, which is finite \cite[Theorem 19.1]{Milne}. We leverage this to prove finiteness or vanishing of certain terms on the $E_2$-page.
\begin{lemma} Let $X$ be a smooth affine fourfold over an algebraically closed field $k$. Then $\Hnis^2(X,\KM_4/2)=\Hnis^3(X,\KM_4/2)=0$.
\end{lemma}
\begin{proof} The $E_2$-terms $\Hnis^i(X, \mathcal{H}^m(j))$ of the Bloch--Ogus spectral sequence are zero for $i >m$ (cf. \cite[Corollary 6.2]{Bloch-Ogus},  \cite[Theorem 19.3]{MVW06}). Furthermore, it follows from \cite[Chapter VI, Theorem 7.2]{Milne} that the restriction of the sheaves $\mathcal{H}^m(j)$ to the small Zariski site of $X$ are zero for $m>4$; in particular $\Hnis^i(X, \mathcal{H}^m(j)) = \Hzar^i(X, \mathcal{H}^m(j)) = 0$ for $m>4$. We depict the $E_2$-page for $X$ in \Cref{fig:bloch-ogus}:
\begin{figure}[H]
\[ \begin{tikzcd}[column sep=.2em]
    {\Hnis^0(X, \sHet^4(j))} & {\Hnis^1(X, \sHet^4(j))} & {\Hnis^2(X, \sHet^4(j))} & {\Hnis^3(X, \sHet^4(j))} & {\Hnis^4(X, \sHet^4(j))}\\
    {\Hnis^0(X, \sHet^3(j))} & {\Hnis^1(X, \sHet^3(j))} & {\Hnis^2(X, \sHet^3(j))} & {\Hnis^3(X, \sHet^3(j))} & 0\\
    {\Hnis^0(X, \sHet^2(j))} & {\Hnis^1(X, \sHet^2(j))} & {\Hnis^2(X, \sHet^2(j))} & 0 & 0\\
    {\Hnis^0(X, \sHet^1(j))} & {\Hnis^1(X, \sHet^1(j))} & 0 & 0 & 0\\
    {\Hnis^0(X, \sHet^0(j))} & 0 & 0 & 0 & 0\\
\end{tikzcd} \]
\centering
\caption{The $E_2$-page for the Bloch--Ogus spectral sequence of a smooth affine variety of dimension at most four, converging to \'etale cohomology with coefficients in $\mu_2^{\otimes j}$. Convergence is along the anti-diagonals.}
\label{fig:bloch-ogus}
\end{figure}

Note that the term \[E_2^{4,4}=\Hnis^4(X, \sHet^4(j))\] is the only possibly nonzero term on the line $p+q=8$, 
converging to $\Het^8(X,\mu_2^{\otimes j})$. This term admits no nontrivial incoming or outgoing differentials, so 
\[E_{\infty}^{4,4}=\Hnis^4(X, \sHet^4(j)).\]
Since $X$ is affine of dimension $4$, $E_\infty^{4,4}\cong \Het^8(X,\mu_2) = 0$ by \cite[Theorem 15.1]{Milne}. 
Similarly, since $\Hnis^3(X, \sHet^4(\mu_2))$ also has no nontrivial incoming or outgoing differentials, it is identical to $E_\infty^{3,4}$ and therefore zero.
Lastly,  $\Hnis^2(X, \sHet^4(\mu_2))$ has no nontrivial incoming or outgoing differentials, hence is equal to $E_\infty^{2,4}$. 
This is a subquotient of $\Het^6(X,\mu_2)$, which is zero by dimension considerations. Therefore $\Hnis^2(X, \sHet^4(\mu_2)) = 0$.
\end{proof}

This completes the proof \Cref{lem:stage1}\Cref{item:H2KM3}.

\subsubsection{Unramified cohomology and the first stage}\label{subsec:c-proof}

We briefly recall {\em unramified cohomology}, defined as those cohomology classes on $X$ that vanish under specialization maps to codimension one points. First, note that the Bloch--Ogus spectral sequence \Cref{eqn:coniveau-SS-E1} with $p=0$ includes a differential
\begin{equation}\label{eq:d1q}
    d_1^q\: \Het^{q}(k(X), \mu_2^{\otimes j}) \to \bigoplus_{\substack{x \in X^{(1)}}}\Het^{q-1}(k(x),\mu_2^{\otimes j-1}).
\end{equation}
\begin{definition}[{\cite[Definition~1.1.1]{CTO}}]
\label{def:unramified-cohomology}
Let $X$ be a smooth connected variety over a field $k$, and assume $n$ is invertible in $k$. The $q$th \textit{unramified cohomology} of $X$ with coefficients in $\mu_2^{\otimes j}$, denoted $\Hnr^q(X,\mu_2^{\otimes j})$, is the kernel of $d_1^q$ from \Cref{eq:d1q}.
\end{definition}
\begin{rmk}\label{rmk:twists-alg-closed} Over an algebraically closed base field, $\mu_2^{\otimes j}\cong \mu_2$. Since we work over an algebraically closed field for all results in this paper, we could in principle omit the tensor powers in our notation. \end{rmk}
By the Bloch--Ogus theorem \cite[Theorem 6.1]{Bloch-Ogus},
\begin{equation}\label{eqn:Hnr-in-terms-of-BO-sheaves}
\begin{aligned}
    \Hnr^q(k(X),\mu_2^{\otimes j}) \cong \Hnis^0(X, \sHet^q(\mu_2^{\otimes j})).
\end{aligned}
\end{equation}
We now have the setup to prove \Cref{lem:stage1}\Cref{item:unr}.

\begin{lemma}\label{lem:fourfold-compare-terms} Let $X$ be a smooth fourfold over an algebraically closed field $k$ of characteristic not equal to $2$. Then $\Hnr^4(k(X),\mu_2^{\otimes j})$ is finite if and only if $\Hnis^2(X,\sHet^3(j))$ is finite.
\end{lemma}
\begin{proof} Consider the Bloch--Ogus spectral sequence for $X$ as in \Cref{fig:bloch-ogus}. The $E_2$-page is depicted in \Cref{fig:bloch-ogus2}.

\begin{figure}[H]
\[ \begin{tikzcd}
    {\Hnis^0(X, \sHet^4(j))}\ar[drr,dashed] & {\Hnis^1(X, \sHet^4(j))}\ar[drr,two heads] & 0 & 0 \\
    {\Hnis^0(X, \sHet^3(\mu_2))} & {\Hnis^1(X, \sHet^3(j))} & {\Hnis^2(X, \sHet^3(j))} & {\Hnis^3(X, \sHet^3(j))} & \\
    \vdots & \vdots & \vdots & \vdots 
\end{tikzcd} \]
\centering
\caption{The $E_2$-page for the Bloch--Ogus spectral sequence of a smooth affine variety of dimension at most four, with differentials indicated.}
\label{fig:bloch-ogus2}
\end{figure}

There are no nontrivial incoming differentials to the term $\Hnis^0(X, \sHet^4(j))$, and the only outgoing differential appears on the $E_2$-page, so we have that the kernel
\begin{equation}\label[empty]{eq:ker1}
    \ker \left( \Hnis^0(X, \sHet^4(j)) \xto{d_2} \Hnis^2(X, \sHet^3(j)) \right)
\end{equation}
is a subquotient of $\Het^4(X,\mu_2^{\otimes j})$, so the kernel \Cref{eq:ker1} is finite by finiteness of \'etale cohomology \cite[Theorem 19.1]{Milne}. The cokernel of this map is zero as it is a subquotient of $\Het^5(X,\mu_2^{\otimes j}) \cong 0$.
Together, we have that the map $d_2\:\Hnis^0(X, \sHet^4(j)) \to \Hnis^2(X, \sHet^3(j))$ is surjective with finite kernel, therefore one of these groups is finite if and only if the other is.
\end{proof}

In particular, this proves \Cref{lem:stage1}\Cref{item:unr} and we conclude that a sufficient condition to have finitely many lifts of prescribed Chern classes to $\BGL_2^{(2)}$ as in \Cref{diag:lift2} is that the unramified cohomology $\Hnr^4(X,\mu_2^{\otimes4})$ is finite. We now study settings where these constraints on unramified cohomology are satisfied.

\begin{proposition}\label{prop:nr-compactify-text}
\label{prop:nr-compactify1} Suppose that $X$ is smooth affine of dimension $4$ over a field $k$, and $X \hookrightarrow \bar{X}$ is a compactification with boundary divisor $D$ that is smooth and irreducible of dimension $3$.  If $\Hnr^4(\bar{X},\mu_2^{\otimes j})$ and $\Hnr^3(D,\mu_2^{\otimes j-1})$ are finite (resp., zero), then so is $\Hnr^{4}(X,\mu_2^{\otimes j})$.
\end{proposition}

\begin{proof}
For any smooth $k$-variety $U$ and integer $j \geq 0$, consider the complex $E^{\ast,4}_{1}(U,j)$  on the $E^1$-page of the Bloch--Ogus spectral sequence associated to $U$. 
The degree zero cohomology of this complex is the fourth unramified cohomology group of $U$ with coefficients in $\mu_{2}^{\otimes j}$
 (cf. \Cref{eqn:coniveau-SS-E1}). The open immersion $X \to \bar{X}$ induces an epimorphism of complexes of abelian groups $E^{\ast,4}_{1}(\bar{X},4) \to E^{\ast,4}_{1}(X,4)$ 
whose kernel is precisely the complex $E^{\ast,3}_{1}(D,3)$ used in the definition of the third unramified cohomology of $D$ with coefficients in $\mu_{2}^{\otimes 3}$ shifted by $1$ degree. Now the long exact sequence of cohomology groups associated to the short exact sequence of complexes yields an exact sequence of abelian groups
\[ 0 \to \Hnr^d(\bar{X},\mu_2^{\otimes j})\to \Hnr^d(X,\mu_2^{\otimes j}) \xrightarrow{\partial} \Het^{d-1}(k(D), \mu_2^{\otimes j-1})\]
with
\[\label[diagram]{eq:bounding_image2}\op{Im}(\partial)\subset \op{Ker}\left(d_{1}^{3}\:\Het^{d-1}(k(D),\mu_2^{\otimes j-1}) \to \oplus_{i}\bigoplus_{y \in D^{(1)}} \Het^{d-2}(k(y),\mu_2^{\otimes j-2})\right),\]
where $d_{1}^{3}$ is the differential in $E_{1}^{\ast,3}(D)$ whose kernel computes the group $\Hnr^3(D,\mu_2^{\otimes j-1})$. 
\end{proof}

\begin{rmk}
Localization sequences are available for refined unramified cohomology groups as defined in \cite{Schreieder}, which generalize unramified cohomology groups. Instead of directly using the localization sequence for refined unramified cohomology groups given in \cite[Theorem 1.3]{Kok-Zhou}, we have given a formal reasoning in the proof of \Cref{prop:nr-compactify1} to make this paper more self-contained.
\end{rmk}
\begin{corollary}\label{cor:nr-compactify-alg-closed} With set-up as in \Cref{prop:nr-compactify1}, suppose that additionally $k$ is algebraically closed.  If $\Hnr^4(\bar{X},\mu_2)$ and $\Hnr^3(D,\mu_2)$ are finite (resp., zero), then so is $\Hnr^{4}(X,\mu_2)$.
\end{corollary}

\subsubsection{Examples of smooth affine fourfolds with finite fourth unramified cohomology}\label{subsubsec:finite-unramified} 
\Cref{prop:nr-compactify1} gives some immediate examples of smooth affine fourfolds with finite fourth unramified cohomology. Throughout, let $k$ be an algebraically closed field of characteristic different from $2$ unless otherwise specified.

\begin{proposition}[{\cite[Corollaire 6.2]{CTV12}}] 
\label{prop:CTV-uniruled-threefolds}
Let $Z$ be a smooth projective threefold over an algebraically closed field of characteristic zero. If $Z$ is uniruled then $\Hnr^3(Z,\mu_2)=0$.
\end{proposition}

\begin{example}\label{exa:projective-hypersurface-complements-H4} 
By \cite{kollarmiyaokamori92}, a smooth projective variety $D$ over $\mathbb C$ with ample anticanonical bundle $-K_D$ is rationally connected and hence uniruled. Let $X$ be the complement of any of the following:
\begin{enumerate}
    \item a smooth hypersurface $Z \subseteq \P^4$ of degree $d\le 4$
    \item a smooth hypersurface $Z \subseteq \P^1 \times \P^3$ of bidegree $(a,b)$ with $a\le1$ and $b\le3$
    \item a smooth hypersurface $Z \subseteq \P^2 \times \P^2$ of bidegree $(a,b)$ with $a\le 2$ and $b\le 2$.
\end{enumerate}
In all these cases above, $\Hnr^3(Z,\mu_2)$ vanishes by \Cref{prop:CTV-uniruled-threefolds} and hence $\Hnr^4(X,\mu_2) = 0$ by \Cref{prop:nr-compactify1}.
\end{example}

\subsection{The second Moore--Postnikov stage}\label{sec:stage2}

We now consider conditions under which the number of $\mathbb A^1$-homotopy classes of lifts of a given morphism $X\to \BGL_2^{(2)}$ to $\BGL_2^{(3)}$ are finite. Our main result is the following:
\begin{theorem}\label{thm:stage2-secintro} Let $X$ be a smooth affine fourfold over an algebraically closed field $k$ of characteristic not equal to $2$ or $3$. Given a fixed 
homotopy class $X \to \BGL_2^{(2)}$, lifts to $\BGL_2^{(3)}$ are finite (resp. unique) if $\Ch^3(X)$ is finite (resp. zero).
\end{theorem}
From the discussion in \Cref{subsec:def-twists}, the number of lifts is a torsor for a quotient of \[\Hnis^3(X,\piA{3}{\BGL_2}(\mathcal L))\cong \Hnis^3(X,\piA{2}{\A^2\minus 0}(\mathcal L)),\] where $\mathcal L$ the line bundle determined by the composite $X\to \BGL_2^{(2)} \to B\mathbb G_m$. 

We start from the observation that 
$\SL_2\simeq \A^2\setminus 0 \simeq \Sp_2$ and use standard symplectic fiber sequences to understand $\piA{2}{\A^2\minus 0}$. We begin in \Cref{sseq-spheres} by recalling results from \cite{AF-3fold} in the language of spectral sequences. In \Cref{subsec:vanishing_km}, we prove vanishing results for cohomology with coefficients in $\KM_i/r$ for $r$ composite. In \Cref{subsec:stage2-proof}, we put together the proof of \Cref{thm:stage2-secintro}.

\subsubsection{A symplectic spectral sequence and unstable homotopy of spheres}\label{sseq-spheres}We give preliminaries that will be used both in the proof of \Cref{thm:stage2-secintro} and again in \Cref{sec:stage3}.

A way to access the unstable $\mathbb A^1$-homotopy sheaves of even punctured affine spaces $\A^{2n}\minus 0$ is via the the natural filtration 
\[ \cdots \to \Sp_{2n}\to \Sp_{2n+2} \cdots \to \Sp,\]
since subsequent terms in the filtration fit into fiber sequences
   $ \Sp_{2n} \to \Sp_{2n+2} \to \A^{2n+2}\minus0.$
For any integer $n \geq 0$, the punctured affine space $\A^{n+2}\minus0$ have first non-trivial homotopy sheaf $\piA{n+1}{\A^{n+2}\minus 0} \cong \K^{MW}_{n+2}$ \cite[Theorem 6.40]{Morel}. Following \cite[Section 2.2]{AF-spectralsequence}, the spectral sequence associated to the filtration above then takes the form
\begin{equation}\label[sseq]{eqn:Sp-SS}
\begin{aligned}
    E_1^{s,t}=\piA{s}{\A^{2t}\setminus 0} \implies \piA{s}{\Sp}.
\end{aligned}
\end{equation}
A differential $d_r$ on the $r$-th page has bidegree $(-1,-r)$, i.e., is a morphism $d_r^{s,t} \colon E_r^{s,t} \to E_r^{s-1,t-r}$.
\begin{rmk}
Note that our grading convention differs from the one in \cite[Section 2.2]{AF-spectralsequence}; by our grading convention, convergence occurs in columns. \end{rmk}
In \Cref{fig:Sp-seq} we write out relevant terms of the first page of \Cref{eqn:Sp-SS}. All terms to the left, below, and above the region shown in \Cref{fig:Sp-seq} are zero.
\begin{figure}[H]
\[\begin{tikzcd}[column sep=.1em]
  0 & 0 & 0 & 0 & \pia_5(\A^6\setminus 0)\arrow[dl,"d_1^{5,3}" description] \arrow[ddl,dashed]  \\
   0 & 0 & \pia_3(\A^4\setminus 0)\arrow["d_1^{3,2}" description]{dl} & \pia_4(\A^4\setminus 0)\arrow["d_1^{4,2}" description]{dl} & \pia_5(\A^4\setminus 0)\ar[dl,"d_1^{5,2}" description] \\
     \pia_1(\A^2\setminus 0) & \pia_2(\A^2\setminus 0) & \pia_3(\A^2\setminus 0) & \pia_4(\A^2\setminus 0) & \pia_5(\A^2\setminus 0)  \\
\end{tikzcd} \]
\centering
\caption{The $E_2$-page for the spectral sequence arising from the dimension filtration on $\Sp$. We depict $d_1$-differentials as solid arrows and a single $d_2$-differential as a dashed arrow.}
\label{fig:Sp-seq}
\end{figure}

\begin{rmk}\label{rmk:sseq-gm-action}  \Cref{eqn:Sp-SS} is a spectral sequence of strictly $\A^1$-invariant sheaves with $\mathbb G_m$-action as follows. By \cite[pp. 2577-2578]{AF-3fold}, there is an action of $\mathbb G_m$ on $\Sp_{2n}$ that
\begin{enumerate}
\item induces a $\mathbb G_m$-action on $\A^{2n}\setminus 0 \simeq \Sp_{2n}/\Sp_{2n-2}$. This action is given by $t\cdot (a_1,\ldots, a_{2n}) = (a_1,t^{-1}a_2, a_3, t^{-1}a_4,\ldots, a_{2n-1}, t^{-1}a_{2n})$;
\item stabilizes to the natural action of $\mathbb G_m$ on $\KSp_i\cong \GW_i^2$ given by composing the canonical map $\mathbb G_m \to \KMW_0 \cong \GW^0_0$ with the action of $(\GW_0^0)^\times$ on $\GW_i^2$ (see \cite[Lemma 4.7]{AF-3fold}). This $\mathbb G_m$-action agrees with the canonical $\mathbb G_m$-action on $\GW_i^2$ as a contraction as explained in Example \ref{ex:contract-act} (see \cite[Lemma 4.6]{AF-3fold}); and
\item induces the $\pia_1$-action of $\mathbb G_m \cong \piA{1}{\BGL_2}$ on $\piA{n}{\BGL_2}$ under the natural identification
\[\begin{tikzcd}
\piA{n}{\BSp_2}&\piA{n}{\BSL_2}\ar[r,"\cong"]\ar[l,"\cong" above]& \piA{n}{\BGL_2}
\end{tikzcd}
\]for $n\geq 2$.
\end{enumerate}
\end{rmk}
In addition to differentials, there are a few other morphisms that will be relevant. The \textit{edge morphisms} $e_s \colon E_{1}^{s,1} \to E_\infty^{s,1}$ take the following form:
\[
    e_s \colon \pi_s(\A^2\minus0) \to \pi_s(\Sp) = \KSp_{s+1}.
\]
The map from $\Sp \to \Sp/\Sp_{2n-2}$ induces a morphism
\begin{align*}
    i_{2n-1} \:\piA{2n-1}{\Sp}\cong \KSp_{2n} \to \piA{2n-1}{\Sp/\Sp_{2n-2}}\cong \piA{2n-1}{\A^{2n}\minus 0}.
\end{align*}
Note again that $\piA{2n-1}{\A^{2n}\minus 0} \cong \KMW_{2n}$ by \cite[Theorem 6.40]{Morel}. Under this identification, the morphism agrees with the map
   $ \phi_{2n} \colon \KSp_{2n} \to \KMW_{2n}$ of \cite[\S3]{AF-3fold}. 
\begin{definition} Following \cite[p.~2568]{AF-3fold}, let
$ \Tsheaf{2n} = \coker \left( \KSp_{2n} \xto{\phi_{2n}} \KMW_{2n} \right).$
\end{definition}

As a particular application of our spectral sequence, we can investigate $\piA{2}{\A^2\minus0}$ via its outgoing edge map $e_2$ and its incoming differential, and we obtain a short exact sequence
\[
    0 \to \im(d_1^{3,2}) \to \piA{2}{\A^2\minus0} \xto{e_2} \KSp_3 \to 0.\]
Note that $e_2$ is an epimorphism because $\piA{2}{\A^2\minus0}$. Moreover, convergence implies that there is an exact sequence
\[
    \KSp_4 \xto{\phi_4} \piA{3}{\A^4\minus 0} \xto{d_1^{3,2}} \piA{2}{\A^2\minus 0} \xto{e_2} \KSp_3 \to 0.
\]
Thus, we can identify
$    \im(d_1^{3,2}) \cong  \piA{3}{\A^4\minus 0}/\ker(d_1^{3,2}) \cong \coker(\phi_4) = \Tsheaf{4}.$
This yields the following result.

\begin{theorem}[{\cite[Theorem~3 and Corollary~4.9]{AF-3fold} }]\label{thm:AF3} There is an exact sequence
\begin{equation}\label[sequence]{eq:21}
0 \to \Tsheaf{4} \to \piA{2}{\A^2\minus 0} \to \KSp_3 \to 0
\end{equation} 
of strictly $\A^1$-invariant sheaves with $\mathbb G_m$-actions. The $\mathbb G_m$-action on $\Tsheaf{4}$ is trivial.
\end{theorem}
Moreover, Asok--Fasel show that there is an exact sequence
\begin{equation}\label[sequence]{eqn:T4S4}
\begin{aligned}
    \I^5 \to \Tsheaf{4} \to \Ssheaf{4} \to 0,
\end{aligned}
\end{equation}
where
\[
    \Ssheaf{4} = \coker (\KSp_{2n+2} \to \KMW_{2n+2} \tto \KM_{2n+2}).
\]
By \Cref{prop:Ij-vanishing}\Cref{a-Ij-vanishing}, $\I^5|_X\cong 0$ for $X$ a smooth affine fourfold. We deduce:

\begin{corollary}[\cite{AF-3fold}]\label{cor:TS4} Let $X$ be a smooth affine fourfold over an algebraically closed field of characteristic $\neq 2$. Then $\Ssheaf{4}|_X \cong \Tsheaf{4}|_X$.
\end{corollary}

By \cite[3.2]{AF-3fold}, we have an epimorphism $
    \KM_4/12 \tto \Ssheaf{4}$ that becomes an isomorphism after 2-fold contraction. By \Cref{prop:gersten-complex-properties}\Cref{contract-isom}:
\begin{corollary}[\cite{AF-3fold}]\label{TSKM4} Let $X$ be a smooth affine fourfold over an algebraically closed field of characteristic $\neq 2$. Then $\Hnis^i(X,\Tsheaf{4}) \cong \Hnis^i(X,\KM_4/12)$ for $i=3,4$.
\end{corollary}

Putting these facts together, consider $X$ a smooth affine fourfold and $\mathcal L$ a line bundle over $X$. We obtain an exact sequence on cohomology
\begin{equation}\label[sequence]{LES008}
 \Hnis^3(X, \KM_4/12) \to \Hnis^3(X,\piA{2}{\A^2\minus0}(\mathcal L)) \xto{{e_2}_*} \Hnis^3(X,\KSp_3(\mathcal L)) \to 
\Hnis^4(X, \KM_4/12) 
\end{equation}
We will first show that $\Hnis^3(X, \KM_4/12)$ and $\Hnis^4(X, \KM_4/12)$ are zero in \Cref{subsec:vanishing_km}. Then, in \Cref{prop:h3k3sp-finite-if-ch3-is}, we relate finiteness (resp. triviality) of $\Hnis^3(X,\KSp_3(\mathcal L))$ to finiteness (resp. triviality) of $\Ch^3(X)$. This will complete the proof of \Cref{thm:stage2-secintro}.

\subsubsection{Vanishing of mod $n$ Milnor $K$-theory cohomology}\label{subsec:vanishing_km}

In this section, we prove vanishing results for cohomology of a smooth affine variety $X$ of dimension $d$ over an algebraically closed fields with coefficients in $\KMW_d/n$, where $d \geq 4$ and $n$ is prime to the characteristic of the base field. This is a key step in the proof of \Cref{thm:stage2-secintro}. 
\begin{definition}\label{mult-n-stuff}Let $\cdot  n\: \KM_d \to \KM_d$ denote multiplication by $n$. Let $\KM_d[n]$ denote the kernel of $\cdot n: \KM_d \to \KM_d$ and $n\KM_d$ denote its image. \end{definition}
We will use the short exact sequences of sheaves
\begin{equation}\label[sequence]{divide-SES-1} 0 \to \KM_d[n] \to \KM_d \to n\KM_d \to 0,
\end{equation}
and
\begin{equation}\label[sequence]{divide-SES-2} 0 \to n\KM_d \to \KM_d \to \KM_d/n \to 0.
\end{equation} 
For each $i$, \Cref{divide-SES-1} gives rise to an exact
\begin{equation}\label[sequence]{divide-LES-1}  \Hnis^{i-1}(X,n\KM_d) \xto{} \Hnis^{i}(X,\KM_d[n]) \xto{a_i} \Hnis^{i}(X,\KM_d) \xto{a'_i} \Hnis^{i}(X,n\KM_d) \xto{} \Hnis^{i+1}(X,\KM_d[n]) \end{equation}
while \Cref{divide-SES-2} gives rise to an exact sequence
\begin{equation}\label[sequence]{divide-LES-2}  \Hnis^{i-1}(X,\KM_d/n) \to \Hnis^{i}(X,n\KM_d) \xto{b_i} \Hnis^{i}(X,\KM_d) \xto{b'_i} \Hnis^{i}(X,\KM_d/n) \to \Hnis^{i+1}(X,n\KM_d) \end{equation}
for each $i$.
\begin{lemma}\label{lem-1} If $X$ is a smooth affine $d$-fold over an algebraically closed field $k$ and $n \geq 2$ is a positive integer that is invertible in $k$, then the natural map
\[b_d\:\Hnis^{d}(X,n\KM_d) \to \Hnis^{d}(X,\KM_d)\] is an isomorphism.
\end{lemma}
\begin{proof}We refer to \Cref{divide-LES-1} and \Cref{divide-LES-2} for notation. By Theorem \ref{thm:chow-divisibility}, the group $\Hnis^d(X,\KM_d)\cong \CH^d(X)$ is divisible and torsion-free. This implies $b_d \circ a_d'$ is an isomorphism and $a_d'$ is injective. The map $a_d'$ is moreover surjective and thus an isomorphism since $\Hnis^{d+1}(X,\KM_d[n])$ is zero by dimension considerations. Thus $b_d$ is an isomorphism as was to be shown.
\end{proof}
Next, we prove a similar result for the $(d-1)$-st cohomology.
\begin{lemma}\label{Tariq-lem1}   If $X$ is a smooth affine $d$-fold over an algebraically closed field $k$ and $n \geq 2$ is a positive integer that is invertible in $k$, then the natural map
\[b_{d-1}\:\Hnis^{d-1}(X,n\KM_d) \to \Hnis^{d-1}(X,\KM_d)\] is an isomorphism.
\end{lemma}
\begin{proof}  By \Cref{lem:fasel-KMW-divisibility}, multiplication by $n$ is an isomorphism on $\Hnis^{d-1}(X,\KMW_d)$. Using \Cref{eqn:M-MW-I-sequence} and \Cref{prop:Ij-vanishing}\Cref{a-Ij-vanishing}, the same is true for multiplication by $n$ on $\Hnis^{d-1}(X,\KM_d)$. Therefore $b_{d-1}\circ a_{d-1}'$ is an isomorphism and $b_{d-1}$ is surjective.

Since $a'_{d-1}$ is injective, it suffices to show that $a'_{d-1}$ is also surjective (and hence bijective). This follows from \Cref{divide-SES-1} with $i=d-1$ and the fact that $\Hnis^{d}(X, K^M_d [n]) = 0$.
\end{proof}

\begin{lemma}\label{tariq-lem0} If $X$ is a smooth affine $d$-fold over an algebraically closed field $k$ and $n \geq 2$ is a positive integer that is invertible in $k$, then $\Hnis^{d-1}(X,\KM_d/n)=0$.
\end{lemma}
\begin{proof} Consider \Cref{divide-LES-2} with $i=d-1$. Surjectivity of $b_{d-1}$ implies that $b_{d-1}'$ is zero. However, since $b_d$ is an isomorphism, $b_{d-1}'$ is also surjective and $\Hnis^{d-1}(X,\KM_d/n)$ is zero.
\end{proof}

\begin{lemma}\label{tariq-lem2}  If $X$ is a smooth affine $d$-fold over an algebraically closed field $k$ and $n \geq 2$ is a positive integer that is invertible in $k$,then $\Hnis^{d-2}(X,\KM_d/n)=0$. \end{lemma}
\begin{proof}This argument is similar
 to the proof of \Cref{tariq-lem0}. \Cref{lem:fasel-KMW-divisibility} implies that $b_{d-2}\circ a_{d-2}'$ is surjective and $b_{d-2}'$ is zero. By \Cref{Tariq-lem1}, $b_{d-1}$ is an isomorphism. So we conclude that $b_{d-2}'$ is surjective and zero. \end{proof}

From \Cref{LES008}, we conclude:
\begin{corollary}\label{cor:stage-2-reduces-to-GW} 
If $X$ is a smooth affine fourfold over an algebraically closed field $k$ of characteristic $\ne 2,3$, then $ \Hnis^2(X, \KM_4/12)=0$, $ \Hnis^3(X, \KM_4/12)=0$, and
 \[{e_2}_*\: \Hnis^3(X,\piA{2}{\A^2\minus0}(\mathcal L)) \to  \Hnis^3(X,\KSp_3(\mathcal L))\] is an isomorphism.
\end{corollary}

\subsubsection{Proof of \Cref{thm:stage2-secintro}}\label{subsec:stage2-proof}

By \Cref{cor:stage-2-reduces-to-GW}, finiteness (resp. uniqueness) of the lifts in stage 2 reduces to finiteness (resp. triviality) of the group $ \Hnis^3(X,\KSp_3)$. We now give conditions for finiteness (resp. triviality) of $ \Hnis^3(X,\KSp_3)$.

\begin{proposition}\label{prop:h3k3sp-finite-if-ch3-is} 
Let $X$ be a smooth affine fourfold over an algebraically closed field, and suppose that $\Ch^3(X)$ is finite (resp. zero). Then $ \Hnis^3(X, \KSp_3)$ is finite (resp. zero).
\end{proposition}
\begin{proof} By \Cref{GW-KSp}, we have an isomorphism $\KSp_3 \simeq \GW_3^2$. Applying \Cref{GW-SES}, we have an exact sequence
\begin{equation}\label[sequence]{seq:sq2-L}
    \Ch^2(X) \to \Ch^3(X) \to  \Hnis^3(X,\GW^2_3(\mathcal L)) \to 0.
\end{equation}
In particular, if $\Ch^3(X)$ is finite (resp. zero), then so is $ \Hnis^3(X,\GW_3^2(\mathcal L))$.
\end{proof}
\begin{remark}\label{rmk:steenrod-sq} Recall that the first map in \Cref{seq:sq2-L} is the twisted Steenrod square $\Sq^2_{\mathcal L}\: \Ch^2(X) \to \Ch^3(X)$; see \cite[Theorem 4.17]{AF-3fold}. So $\Hnis^3(X,\GW^2_3(\mathcal L))\cong \coker(\Sq^2_{\mathcal L})$.
\end{remark}
Combining \Cref{cor:stage-2-reduces-to-GW} and \Cref{prop:h3k3sp-finite-if-ch3-is} complets the proof of \Cref{thm:stage2-secintro}.

\subsection{The third Moore--Postnikov stage}\label{sec:stage3}

The number of lifts from the second stage to the third is bounded by $ \Hnis^4(X, \piA{4}{\BGL_2}(\mathcal L))$. Note that $\piA{4}{\BGL_2} \cong \piA{3}{\SL_2}\cong \piA{3}{\A^2\minus 0}$. We will prove:  
\begin{theorem}\label{thm:stage3}
Let $X$ be a smooth affine fourfold over an algebraically closed field of characteristic not equal to $2$ or $3$. Let $\mathcal L$ be a line bundle on $X$. Then $ \Hnis^4(X,\piA{3}{\A^2\setminus 0}(\mathcal L))=0$.
\end{theorem}
We note a first reduction.
\begin{lemma}\label{thm:stage3-red} With set-up as in \Cref{thm:stage3}, suppose that $ \Hnis^4(X,\piA{3}{\A^2\setminus 0})=0$. Then $ \Hnis^4(X,\piA{3}{\A^2\setminus 0}(\mathcal L))=0$.
\end{lemma}
\begin{proof}
This follows from \cite[Lemma 2.2.3]{fasel2021suslins}.
\end{proof}
To prove that $ \Hnis^4(X,\piA{3}{\A^2\setminus 0})=0$ will require more work. We give background in \Cref{subsec:prelim3}. We then prove \Cref{thm:stage3} in \Cref{thm-stage-3-proof}.

\subsubsection{Preliminaries}\label{subsec:prelim3}
Recall \Cref{eqn:Sp-SS}, from which we obtain an exact sequence involving edge maps and differentials:
\begin{equation}\label[sequence]{eqn:complex-from-ES}
\begin{aligned}
    \cdots \to \piA{4}{\A^4\minus 0} \xto{d_1^{4,2}} \piA{3}{\A^2\minus 0} \xto{e_3} \KSp_4 \xto{\phi_4} \piA{3}{\A^4\minus0} \xto{d_{1}^{3,2}} \piA{2}{\A^2\minus 0} \to \cdots
\end{aligned}
\end{equation}

\begin{definition}\label{def:CD} Let $C=\coker(d_1^{4,2}) $ and $D= \ker(d_1^{3,2})$.\end{definition}

Considering \Cref{eqn:complex-from-ES} at the term $\KSp_4$, we get a short exact sequence
\[
    0 \to C \to \KSp_4 \to D\to 0.
\]
\begin{remark} By exactness, we have identifications
\begin{align*}
    C &\cong \piA{3}{\A^2\minus 0}/\ker(e_3) \\ &  \cong \ker(\phi_4) \\ & \cong \im(e_3), \\
    D &\cong \im(\phi_4). 
\end{align*}
\end{remark}

We therefore obtain the following diagram, where both vertical sequences and the diagonal sequence are short exact:
\begin{equation}\label[diagram]{eqn:LES-H4C}
\begin{tikzcd}[column sep=8mm, row sep=8mm]
    & 0\dar& & 0\dar & 0\\
    & \im(d_1^{4,2})\dar& & D\dar\ar[ur] \\
    & \piA{3}{\A^2\minus0}\dar\rar[dashed,"e_3" above] & \KSp_4\ar[ur]\rar[dashed,"\phi_4"] & \KMW_4\dar  \\
    & C\dar\ar[ur, crossing over]& & \Tsheaf{4}\dar \\
    0\ar[ur] & 0 & &  0 
\end{tikzcd}
\end{equation}

\subsubsection{Proof of \Cref{thm:stage3}}\label{thm-stage-3-proof}
We now prove the main theorem of this section; this involves the sheaves $C$ and $D$ introduced in \Cref{def:CD} and their relationship to the cohomology group of interest. 

\begin{lemma}\label{lem:red2} The map of sheaves $\piA{3}{\A^2\setminus 0} \to C$ induces an isomorphism
\begin{align*}
     \Hnis^4(X,\piA{3}{\A^2\setminus 0}) \xto{\sim}  \Hnis^4(X,C),
\end{align*}
for $X$ a smooth affine fourfold over a field with $2$ and $3$ invertible.
\end{lemma}
\begin{proof} By the leftmost vertical short exact sequence in \Cref{eqn:LES-H4C}, and the fact that $X$ has dimension four, we obtain an exact sequence of cohomology groups
\begin{equation}\label{eq:exactness1}  \Hnis^4(X, \im(d_1^{4,2})) 
\to  \Hnis^4(X, \piA{3}{\A^2\setminus 0})
\to  \Hnis^4(X, C)
\to 0,
\end{equation}
hence it suffices to argue that $ \Hnis^4(X,\im(d_1^{4,2})) = 0$. We consider the short exact sequence of sheaves
\begin{align*}
    0\to \ker(d_1^{4,2}) \to \piA{4}{\A^4\setminus 0} \to \im(d_1^{4,2})\to 0,
\end{align*}
giving rise to a long exact sequence
\begin{equation}\label[sequence]{eq:exactness2}
\cdots \to  \Hnis^4(X, \ker(d_1^{4,2})) \to  \Hnis^4(X,\piA{4}{\A^4\setminus 0}) \to  \Hnis^4(X,\im(d_1^{4,2}))\to 0.
\end{equation}
Since $X$ is a smooth affine fourfold, $ \Hnis^4(X, \piA{4}{\A^4\setminus0}) = 0$ by \cite[Theorem 3.0.3]{fasel2021suslins}. Together with \Cref{eq:exactness2}, this implies that $ \Hnis^4(X,\im(d_1^{4,2})) = 0$ and the result follows.
\end{proof}
Combining \Cref{eqn:LES-H4C} and \Cref{lem:red2}, we obtain
\begin{corollary}\label{rmk:H4C} There is an exact sequence
\begin{equation}\label[sequence]{eq:LESCD} 
\cdots \to  \Hnis^3(X,D) \to  \Hnis^4(X, \piA{3}{\A^2\minus 0}) \to  \Hnis^4(X, \KSp_4)\to \cdots\, .
\end{equation}
\end{corollary}

To prove \Cref{thm:stage3} it suffices to argue that both the incoming map from $ \Hnis^3(X,D)$ and the outgoing map to $ \Hnis^4(X,\KSp_4)$ in \Cref{eq:LESCD} factor through zero. 
We next study the term $ \Hnis^4(X,\KSp_4)$. Recalling that $\GW_4^2\cong \KSp_4$ by 
\Cref{GW-KSp}, the following general result is applicable:

\begin{proposition}\label{prop:HnGWn2n}
Let $X$ be a smooth affine $n$-fold over a field $k$ of characteristic not equal to $ 2$. Then $$ \Hnis^n(X,\GW^{n-2}_n)\cong \CH^n(X)$$
\end{proposition}
\begin{proof}
Let $F$ be a field. By \cite[Lemma 2.3]{Fasel-Rao-Swan} the hyperbolic map $$h\:\K_i^Q(F) \to \GW^{i-2}_i(F)\simeq \GW^{i+2}_i(F)$$ is an isomorphism for $i=0,1.$ Consider Rost--Schmid complexes computing $ \Hnis^n(X,\GW^{n-2}_n)$ and $ \Hnis^n(X,\K_n^Q)$, and the comparison map induced by the hyperbolic map:

\[\begin{tikzcd}
& \oplus_{x \in X^{(n-1)}}\K_{1}^M(k(x))\ar[d,"\simeq"] 
& \oplus_{x \in X^{(n)}}\K_{0}^M(k(x)) \ar[d,"\simeq"] 
\\
    \cdots\ar[r]
    & \oplus_{x \in X^{(n-1)}}(\K_{n}^Q)_{-(n-1)}(k(x)) \ar[r]\ar[d,"h"]
    & \oplus_{x \in X^{(n)}}(\K_{n}^Q)_{-n}(k(x))\ar[r]\ar[d,"h"]
    &0 
    \\
     \cdots\ar[r]
     & \oplus_{x \in X^{(n-1)}}(\GW^{n-2}_n)_{-(n-1)}(k(x)) \ar[d,"\simeq"]\ar[r]
     & \oplus_{x \in X^{(n)}}(\GW^{n-2}_n)_{-n}(k(x))\ar[r]\ar[d,"\simeq"]
     &0
     \\
      & \oplus_{x \in X^{(n-1)}}\GW^{-1}_{1}(k(x)) 
      & \oplus_{x \in X^{(n)}}\GW^{-2}_0(k(x)) 
      &.\\
\end{tikzcd}\]
So $ \Hnis^n(X,\GW^{n-2}_n)\simeq$ $ \Hnis^n(X,\K_n^Q)$ $\simeq$ $ \Hnis^n(X,\K_n^M)\simeq \CH^n(X)$.
\end{proof}
 We next consider cohomology with coefficients in $D$.
\begin{proposition}\label{prop:H3XD} Let $X$ be a smooth affine fourfold over an algebraically closed field $k$ of characteristic not equal to $ 2$ or $3$. Then the natural maps of sheaves $D \to \KMW_4 \to \KM_4$ induce isomorphisms
\begin{equation}\label[sequence]{sequencesequence009}
     \Hnis^3(X,D) \xto{\sim}  \Hnis^3(X,\KMW_4) \xto{\sim}  \Hnis^3(X,\KM_4).
\end{equation}
\end{proposition}
\begin{proof} Consider the long exact sequence on cohomology induced by the short exact sequence
\[ 0\to D \to \KMW_4 \to \Tsheaf{4} \to 0.\]
Note that $ \Hnis^2(X,\Tsheaf{4}) =0$ and $  \Hnis^3(X,\Tsheaf{4}) = 0$ by \Cref{tariq-lem0,tariq-lem2} together with \Cref{prop:gersten-complex-properties} and the fact that $\K^M_4 \to \textbf{T}'_4$ is an epimorphism that becomes an isomorphism after $2$-fold contraction. 
Thus, \Cref{sequencesequence009} gives an isomorphism $ \Hnis^3(X,D) \xto{\sim}  \Hnis^3(X,\KMW_4)$. 
\end{proof}

Combining \Cref{prop:HnGWn2n}, \Cref{prop:H3XD}, and \Cref{eq:LESCD} gives us the following result:
\begin{corollary}\label{eqn:SES-stage-3-revised} Let $X$ be a smooth affine fourfold over an algebraically closed field of characteristic not equal to $2$ or $3$. There is an exact sequence
\[\begin{aligned}
    \cdots\to  \Hnis^3(X,\KSp_4) \to  \Hnis^3(X,\KM_4) \to  \Hnis^4(X,\piA{3}{\A^2\minus0}) \to \CH^4(X) \to \cdots
\end{aligned}\]
\end{corollary}

\begin{lemma}\label{torsion_sheaves} Let $X$ be a smooth affine fourfold over an algebraically closed field $k$. Then \[ \Hnis^4(X,\piA{3}{\SL_2})\otimes \Q =  \Hnis^3(X,\piA{3}{\SL_2}) \otimes \Q=0.\] In particular, all elements in $ \Hnis^4(X,\piA{3}{\SL_2})$ and $ \Hnis^3(X,\piA{3}{\SL_2})$ are torsion.
\end{lemma}
\begin{proof} Since tensoring with $\mathbb Q$ is exact, by considering the Rost--Schmid complex for computing cohomology of a strictly $\mathbb A^1$-invariant sheaf $\textbf{A}$ we find that $\Hnis^i(X,\textbf{A}) \otimes \mathbb Q \simeq \Hnis^i(X,\textbf{A} \otimes \mathbb Q)$. So, it suffices to prove that \[ \Hnis^3\left(X,\piA{3}{\SL_2}\otimes \mathbb Q\right)= \Hnis^4\left(X,\piA{3}{\SL_2})\otimes \mathbb Q\right)=0.\]
By \cite[Theorem 5.3.3]{loc-nilp}, loops on $c_2$ induces a rational $\mathbb A^1$-weak equivalence
$$\Omega c_2\: \SL_2 \to K(\Z(2),3).$$
By \cite[4.3.10]{loc-nilp} the induced maps 
$$\pia_i(\SL_2) \otimes \mathbb Q \to \pia_i\left( K(\Z(2),3)\right)\otimes \Q$$ are isomorphisms of sheaves. 
We note that \[ \Hnis^3(X,\piA{3}{K(\Z(2),3)}= \Hnis^4(X,\piA{3}{K(\Z(2),3)})=0,\] so the same is true with rationalized coefficients. To see this, recall that $\pi_k(K(\Z(j),i))\cong \mathcal H^{i-k,j},$ where $\mathcal H^{i,j}$ is the Nisnevich sheafification of the presheaf $U \mapsto \Hmot^{i}(U,\Z(j))$ on $\Sm_k$. In particular, $\pi_3(K(\Z(2),3)) \simeq \mathcal H^{0,2} $. The contractions of the sheaves $\mathcal H^{i,j}$ have a simple formula: by the Tate suspension isomorphism for motivic cohomology, one has
\[(\mathcal H^{i,j})_{-m}\cong\mathcal H^{i-m,j-m}.\]
Thus the term at degree $m$ in the Rost--Schmid complex computing cohomology with coefficients in $\piA{3}{K(\Z(2),3)}$ takes the form
$$\bigoplus_{x\in X^{(m)}} \mathcal H^{-m,2-m}(k(x))\otimes \Q.$$
If $m\geq 3$, then $2-m<0$ and each cohomology group $\Hmot^{-m}(k(x),\Z(2-m))$ is zero for all points $x$, so there is no cohomology at that term in the complex.
\end{proof}
By \Cref{thm:chow-divisibility} the group $\CH^4(X)$ is torsion-free. Using this observation, we obtain the following result by combining \Cref{eqn:SES-stage-3-revised} and \Cref{torsion_sheaves}:
\begin{corollary} For $X$ a smooth affine fourfold over an a field with $2$ and $3$ invertible, there is a surjection $\Hnis^3(X,\KM_4) \to \Hnis^3(X,\piA{3}{\A^2\setminus 0})$.
\end{corollary}

So, to complete the proof of \Cref{thm:stage3}, it is enough to prove the next proposition.

\begin{proposition}\label{zeromap} The surjective map $ \Hnis^3(X,\KM_4) \to  \Hnis^4(X,\piA{3}{\A^2\setminus 0})$ is zero.
\end{proposition} 
\begin{proof} Consider the map of sheaves
\begin{equation}\label[empty]{eq:composite0101}\KSp_4 \xrightarrow{\phi_4} \KMW_4 \to \KM_4,\end{equation}
where $\phi_4$ is as in \Cref{eqn:complex-from-ES} and $\KMW_4 \to \KM_4$ is the natural map.
The composite \Cref{eq:composite0101} induces the first map in the exact sequence from \Cref{eqn:SES-stage-3-revised}. We will prove that the map induced map
\begin{equation}\label[empty]{eq:small-composite} \Hnis^3(X,\KSp_4) \to  \Hnis^3(X,\KM_4),\end{equation} which occurs in the exact sequence from \Cref{prop:H3XD}, is surjective in characteristic not equal to $2$ or $3$.

To do so, consider the composite
\begin{equation}\label[diagram]{long-composite}\KM_4 \xrightarrow{\mu_4} \K_4^{Q} \xrightarrow{H_{4,2}} \KSp_4 \xrightarrow{\phi_4} \KMW_4 \to \KM_4,\end{equation} where $\K_4^{Q}$ denotes the fourth Quillen $K$-theory sheaf, $\mu_4$ is the map induced from Milnor $K$-theory to Quillen $K$-theory by the natural identification $\KM_1 \to \K_1^Q$, and $H_{4,2}$ is a hyperbolic morphism. For more details on these morphisms, we refer the reader to the discussion preceding \cite[Lemma 4.2]{AF-3fold}. 

We will prove that the long composite in \Cref{long-composite} induces a surjection after applying $ \Hnis^3(X,-),$ which implies that \Cref{eq:small-composite} is also surjective.
By the discussion in the paragraph before \cite[Proposition 3.2]{AF-3fold}, the composite
\[\KSp_4 \to \KMW_4 \to \KM_4\]
coincides with the composite 
\[\KSp_4 \xto{f_{4,2}} \K^Q_4 \xto{\psi_4} \KM_4,\]
where $f_{4,2}$ is the forgetful morphism. So the long composition \Cref{long-composite} is the same as the composite
\begin{equation}\label[diagram]{long-composite2}\KM_4 \xrightarrow{\mu_4} \K_4^{Q} \xrightarrow{H_{4,2}} \KSp_4 \xto{f_{4,2}} \K^Q_4 \xto{\psi_4} \KM_4,\end{equation} 

By \cite[Lemma 4.2]{AF-3fold} the composition \Cref{long-composite2} is the same as the composite
\begin{equation}\label[diagram]{long-composite4}\KM_4 \xrightarrow{\cdot 2} \KM_4 \xrightarrow{\mu_4} \KSp_4 \xto{\psi_4} \KM_4.\end{equation} 

By the paragraph before \cite[Proposition 3.2]{AF-3fold}, $\psi_4 \circ \mu_4$ is multiplication by $6$ on $\KM_4$.
It follows that \Cref{long-composite} is multiplication by $12$ on $\KM_4$ and therefore induces multiplication by $12$ on cohomology. Away from characteristics $2$ and $3$, this map is surjective by \Cref{lem:fasel-KMW-divisibility}. 
\end{proof}

\section{Other cohomological approaches to classification}\label{subsec:euler}

The cohomological investigations in the previous section provide far more insight than is contained in the statement of \Cref{theorem-1}. In particular, we can deduce some classification results for rank $2$ bundles using Euler classes and (stable) symplectic $K$-theory classes. 

In \Cref{coh-conditions}, we observe that vanishing of $\Ch^3(X)$ guarantees a bijection between isomorphism classes of rank $2$ algebraic vector bundles on $X$ determined by their first Chern class and Euler class. Moreover, we show that any choice of determinant bundle and Euler class in an appropriately twisted Chow--Witt group can be achieved as an invariant of some rank $2$ algebraic vector bundle on $X$.  

In \Cref{classify-square} we study rank $2$ vector bundles with fixed determinant bundle $\mathcal L$ which is the square of another line bundle $\mathcal N$. We show that such bundles are determined by an associated (stable) symplectic $K$-theory class. Conceptually, this gives a cohomological classification of such bundles. 

\subsection{Cohomological criteria for the first Chern class and Euler class to uniquely determine a bundle}\label{coh-conditions}
Let $X$ be a smooth affine fourfold over an algebraically closed field $k$ with characteristic not equal to $2$ or $3$ and let $\mathcal L$ be a line bundle over $X$. Given any element \[e \in \CHW^2(X,\mathcal L),\] it is natural to ask whether there exists a rank $2$ bundle $\mathcal{E}$ over $X$ with $\det(\mathcal{E}) \cong \mathcal L$ and Euler class equal to $e$. We prove that this is indeed the case, and then study cohomological conditions under which determinant and Euler class completely classify algebraic vector bundles of rank $2$ over $X$. We refer the reader to \cite[Chapter 8]{Morel} for a detailed introduction to the theory of Euler classes in $\A^1$-homotopy theory.

In this section, we use the Postnikov tower of $\BGL_{2}$ in $\A^1$-homotopy theory (see \Cref{rmk:postnikov} or \cite[Theorem 6.1]{AF-3fold} for details). Abusing notation, we write $\BGL_{2}^{(i)}$ for the $i$-th stage of the Postnikov tower for $\BGL_{2}$. These spaces should not be confused with the $i$-th stages of the Moore--Postnikov factorization of the morphism $(c_{1},c_{2})\: \BGL_2 \to K(\KM_1,1) \times K(\KM_2,2)$ considered in the previous section. 
\begin{rmk}
This notation is justified since, for $i\geq 3$, the $i$-th Postnikov tower stage for $\BGL_2$ agrees with the $i$-th Moore--Postnikov tower stage for $\BGL_2 \to K(\KM_1,1) \times K(\KM_2,2)$. So, in all interesting cases considered below, no problem can arise from the potential ambiguity.
\end{rmk}
Recall that $\BGL_2$ is the homootpy limit of a system \[ \cdots \to \BGL_{2}^{(i+1)} \to \BGL_{2}^{(i)} \to \cdots .\] A morphism $X \to \BGL_{2}^{(i)}$ in $\mathcal H(k)$ can be lifted to a morphism $X \to \BGL_{2}^{(i+1)}$ in $\mathcal H(k)$ if and only if an associated element in $\Hnis^{i+2}(X,\piA{i+1}{\BGL_{2}}(\mathcal L))$ is zero, where $\mathcal{L}$ is the line bundle corresponding to the composite $X \to \BGL_{2}^{(i)} \to \BGL_{2}^{(1)} = B\mathbb{G}_m$. The choices of lifts are acted on transitively by $\Hnis^{i+1}(X,\piA{i+1}{\BGL_{2}}(\mathcal L))$. 
\begin{lemma}\label{lem:lift-third-stage} Let $X$ be a smooth affine fourfold over an algebraically closed field of characteristic not equal to $2$ or $3$. Then $[X,\BGL_2]\cong [X,\BGL_2^{(3)}]$.
\end{lemma}
\begin{proof}
Recall that $[X, \BGL_{2}] = [X,\BGL_{2}^{(4)}]$. Furthermore, since $\Hnis^4(X,\piA{4}{\BGL_{2}}(\mathcal L)) \cong \Hnis^4(X,\piA{3}{\A^2\setminus 0}(\mathcal L)) = 0$ by \Cref{thm:stage3}, any two choices of lift of a given map $X\to \BGL_2^{(3)}$ to $\BGL_2^{(4)}$ are $\A^1$-homotopic. Moreover, the obstruction to lifting vanishes by cohomological dimension considerations. Thus, \[[X, \BGL_{2}] \cong [X,\BGL_2^{(4)}] \cong [X,\BGL_{2}^{(3)}].\]\end{proof}
\begin{lemma}\label{lem:every-class-is-euler-class-of-bundle} 
Let $X$ be a smooth affine fourfold over an algebraically closed field $k$ of characteristic not equal to $2$ or $3$ and let $\mathcal L$ be a line bundle on $X$. Then any cohomology class $e \in \CHW^2(X, \mathcal L)$ is the Euler class of a rank $2$ algebraic vector bundle on $X$ with determinant $\mathcal L$.
\end{lemma}
\begin{proof}
Given \Cref{lem:lift-third-stage}, we must show that lifts of a given map $X \to \BGL_2^{(2)}$ to $\BGL_2^{(3)}$ exist and are homotopically unique. First, consider the obstruction to such a lift. This obstruction takes values in
\[
    \Hnis^4(X,\piA{3}{\BGL_{2}}(\mathcal{L}))= \Hnis^4(X,\piA{2}{\A^2 \setminus 0}(\mathcal{L})).
\]
By the long exact sequence on cohomology associated to \Cref{eq:21} and \Cref{cor:TS4}, it suffices to prove that $\Hnis^4(X,\Tsheaf{4})= \Hnis^4(X,\KSp_{3})=0$.

By \Cref{thm:AF3}, \[ \Hnis^4(X,\Tsheaf{4})\cong  \Hnis^4(X,\KM_4/12)\cong \CH^4(X)/12.\] By \Cref{thm:chow-divisibility}, $\CH^4(X)/12=0$. On the other hand, $\KSp_3 \cong \GW_3^2$. Since $(\GW_3^2)_{-4} = 0$ by \Cref{prop:contractions-1}, the fourth term in the Rost--Schmid complex computing Nisnevich cohomology with coefficients in $\GW_{3}^{2}$ is identically zero.

In particular, we have shown that the map $[X, \BGL_{2}] \to [X, \BGL_{2}^{(2)}]$ is surjective. It follows from \cite[Proposition 6.3]{AF-3fold} that this map associates to a rank $2$ bundle $X \to \BGL_2$ precisely its determinant bundle $\mathcal L$ and its Euler class in $\CH^2(X,\mathcal L)$. This finishes the proof.
\end{proof}

\begin{definition}\label{def:L-or}Let $\mathcal{L}$ be a line bundle over a smooth affine $k$-variety $X$. An $\mathcal{L}$-oriented vector bundle of rank $r$ over $X$ is a pair $(\mathcal{E}, \varphi)$, where $\mathcal{E}$ is a rank $r$ vector bundle over $X$ with an isomorphism $\varphi\: \det (\mathcal{E}) \xrightarrow{\cong} \mathcal{L}$. Two such $\mathcal{L}$-oriented vector bundles $(\mathcal{E}, \varphi)$ and $(\mathcal{E}', \varphi')$ are isomorphic if there is an isomorphism $i\: \mathcal{E} \xrightarrow{\cong} \mathcal{E}'$ such that $\varphi' \circ \det (i)= \varphi$. We let $\Vect^{\mathcal L}_r(X)$ denote the set of isomorphism classes of $\mathcal{L}$-oriented rank $r$ algebraic vector bundles on $X$. \end{definition} 

\begin{rmk} If $\mathcal{L} = \mathcal{O}_X$ in \Cref{def:L-or}, we recover oriented vector bundles in the usual sense.\end{rmk}

\begin{theorem}\label{thm:ch3-zero-determined-euler} Let $X$ be a smooth affine fourfold over an algebraically closed field $k$ of characteristic not equal to $2$ or $3$, and suppose that $\Ch^3(X) = 0$. Then, by taking the Euler class, we obtain a bijection
\begin{equation}\label[empty]{eqn:euler-class-map-1}
    e_{\mathcal L}\: \Vect^{\mathcal L}_2(X) \xto{\sim} \CHW^2(X, \mathcal L)
\end{equation}
for any line bundle $\mathcal{L}$ over $X$. In particular, an algebraic vector bundle $\mathcal{E}$ of rank $2$ on $X$ splits off a trivial vector bundle of rank $1$ if and only if the Euler class of $\mathcal E$ is zero in the Chow--Witt group $\CHW^{2}(X, \det ({\mathcal{E}}))$.
\end{theorem}
\begin{proof}
Surjectivity of \Cref{eqn:euler-class-map-1} follows from \Cref{lem:every-class-is-euler-class-of-bundle}. Up to homotopy, choices of lifts of a given map \[X \to \BGL_2^{(2)}\simeq K^{\mathbb G_m}(\KMW_2,2)\] to $\BGL_2^{(3)}$ are acted on transitively by \[\Hnis^3(X,\piA{3}{\BGL_{2}}(\mathcal{L})) \cong \Hnis^3(X,\piA{2}{\A^2 \minus 0}(\mathcal{L})),\] which is zero by \Cref{cor:stage-2-reduces-to-GW}, \Cref{prop:h3k3sp-finite-if-ch3-is} and by our assumption that $\Ch^3(X)$ is zero. In particular, it follows that under our assumptions the map $[X, \BGL_{2}] \to [X, \BGL_{2}^{(2)}]$ is bijective.
\end{proof}

\begin{remark}\label{rmk:fibers-ch3-finite} If $\Ch^3(X)$ is assumed only to be finite, but not necessarily zero, then $e_{\mathcal L}$ in \Cref{eqn:euler-class-map-1} is a surjection with finite fibers.
\end{remark}

\subsection{Symplectic stabilization}\label{classify-square}
Combining analysis of the Moore--Postnikov tower for $\BSp_2 \to \BSp$ with some of our previous computations, we are able to show that rank $2$ vector bundles on smooth affine fourfolds whose determinant bundle is a square are determined by their symplectic $K$-theory class. More precisely:
\begin{theorem}\label{symp-classification}Let $X$ be a smooth affine fourfold over an algebraically closed field $k$ of characteristic not equal to $2$ or $3$. Let $\mathcal L$ be a line bundle on $X$ such that $\mathcal L \simeq (\mathcal N)^{\otimes 2}$ for some line bundle $\mathcal N$ on $X$. Then there is an injective map
    $$\Vect^{\mathcal L}_2(X) \hookrightarrow \widetilde{\KSp}_0(X)$$ given by taking a vector bundle $\mathcal{V}$ with determinant $\mathcal L$ to the composite \[\left(X \xto{\mathcal V \otimes \mathcal N^{-1}} \BSL_2 \to \BSp\right)\in [X,\BSp].\]
\end{theorem}
We reduce \Cref{symp-classification} to the case that $\mathcal L$ is trivial by noting that tensoring by $\mathcal {N}^{-1}$ induces a bijection
\[\Vect_2^{\mathcal{O}_{X}}(X) \cong \Vect_2^{\mathcal L}(X).\]
Hence the proof of \Cref{symp-classification} is completed by the following:

\begin{theorem}\label{thm:sp-inj}
Let $X$ be a smooth affine fourfold over an algebraically closed field $k$ of characteristic not equal to $2$ or $3$. Then the natural map
\[
    [X,\BSp_2] \to [X,\BSp]
\]
is injective.
\end{theorem}

\begin{proof}
We first consider the morphism $\BSp_{4} \to \BSp$. Its homotopy fiber has its first nonzero $\A^1$-homotopy sheaf in degree $5$, so the obstruction theory involving the Moore--Postnikov factorization associated to the morphism $\BSp_{4} \to \BSp$ implies that $[X, \BSp_{4}] \to [X, \BSp]$ is bijective. In particular, it remains to show that $[X, \BSp_{2}] \to [X, \BSp_{4}]$ is injective.

The relevant stages of the Moore--Postnikov tower for $\BSp_2 \to \BSp_{4}$ take the form below (we refer the reader to \Cref{subsec:def-twists} for background on Moore--Postnikov towers):

\begin{equation}\label[diagram]{diag:MP-BSp}
\begin{tikzcd}
\mathcal E^{(4)}\ar[d]\\
\mathcal E^{(3)} \ar[d]\ar[r,"k_5"] & K(\piA{4}{\A^4 \setminus 0},5) \\
    \BSp_{4} \ar[r,"k_4"] &K(\KMW_4,4) 
\end{tikzcd}
\end{equation}
where $[X,\mathcal E^{(4)}] \cong [X, \BSp_2]$ for $X$ a smooth affine fourfold. Injectivity has to do with uniqueness of choices of lift from $\BSp_{4}$ to $\mathcal E^{(4)}$, which we count one stage at a time.
\begin{itemize}
    \item 
Lifts of a map $y\: X \to \BSp_{4}$ to $\mathcal E^{(3)}$ are a torsor for a quotient of $H^3(X,\KMW_4)$. 
\item Given a lift $\tilde{y}\: X \to \mathcal E^{(3)}$, lifts to $\mathcal E^{(4)}$ are a torsor for a quotient of $H^4(X, \piA{4}{\A^4 \setminus 0}),$ which is zero by \cite[Theorem 3.0.3]{fasel2021suslins}. Thus, $[X,\mathcal E^{(4)}] \cong [X,\mathcal E^{(3)}]$ via the obvious map.
\end{itemize}
 To show lifts of a map $X \to \BSp_{4}$ to $\mathcal E^{(3)}$ are unique, we compare the Moore--Postnikov tower for $\BSp_2 \to \BSp_{4}$ to that for the $\BSL_2 \to K(\KMW_2,2)$, where the latter morphism represents the Euler class (see \Cref{rmk:canonicity} for a discussion of comparison maps Moore--Postnikov towers). Note that we have a commutative diagram

 \begin{equation}\label[diagram]{diag:AB-square}
 \begin{tikzcd}
     \BSp_2 \ar[d,"A" left]\ar[r,"\simeq"]& \BSL_2\ar[d,"B"]\\
     \BSp_{4} \ar[r] & K(\KMW_2,2)
 \end{tikzcd}
 \end{equation}
Taking first stages in the Moore--Postnikov towers for $A$ and $B$ in \Cref{diag:AB-square}, we obtain a commutative diagram:
\begin{equation}\label[diagram]{diag:AB-first-stage}
\begin{tikzcd}
K(\KMW_4,3) \ar[r,"f"] \ar[d,"g"]& K(\piA{2}{A^2 \setminus 0},3)\ar[d] \\
    \mathcal E^{(3)}\ar[d] \ar[r] &\mathcal F^{(3)}\ar[d]\\
    \BSp_{4} \ar[r] & K(\KMW_2,2)
\end{tikzcd}
\end{equation}
where both columns are fiber sequences. Given a basepoint map $X\to \BSp_{4}$ and lift of this basepoint to a map $X\to \mathcal E^{(3)}$, apply $[X, -]$ to \Cref{diag:AB-first-stage} to obtain a commutative diagram
\begin{equation}\label[diagram]{diag:AB-first-stage-htpy}
\begin{tikzcd}
{[X,K(\KMW_4,3)]}
\ar[r,"{f_*}"] \ar[d,"{g_*}"]
 & {[X,K(\piA{2}{A^2 \setminus 0},3)]}
 \ar[d] 
\\
    {[X,\mathcal E^{(3)}]} \ar[d] \ar[r] &{[X,\mathcal F^{(3)}]} \ar[d]\\
    {[X,\BSp_{4}]} \ar[r] & {[X,K(\KMW_2,2)]}.
\end{tikzcd}
\end{equation}
where the columns are exact sequences of pointed sets. The set of lifts of the given basepoint is a torsor for the image of the map labeled $g_*$.

 By \Cref{cor:stage-2-reduces-to-GW} and \Cref{prop:h3k3sp-finite-if-ch3-is}, $H^3(X,\piA{2}{\A^2\setminus 0})$ is a $2$-torsion group, while $H^3(X,\KMW_4)$ is uniquely $2$-divisible by \cite[Theorem 4.0.3]{fasel2021suslins}. Thus $f_*$ is zero. 
 Lastly, note that the center horizontal map in \Cref{diag:AB-first-stage-htpy} is a bijection: we have a string of bijections
 \begin{align*}[X, \mathcal E^{(3)}]\cong [X,\mathcal E^{(4)}] \cong [X,\BSp_2] \cong [X,\mathcal F^{(4)}]  \cong [X,\mathcal F^{(3)}],\end{align*} where the last bijection uses \Cref{thm:stage3}.
Thus $g_*$ factors through zero and its image is zero.
\end{proof}
\begin{rmk} While \Cref{symp-classification} asserts the existence of an injection, we can say a bit more. The proof of \Cref{thm:sp-inj} shows that, for as in the statement of the theorem \Cref{symp-classification}, $\Vect^{\mathcal L}_2(X)$ is in bijection with the elements in $[X,\BSp]\cong [X,\BSp_4]$ such that $k_4(\mathcal E \otimes \mathcal N^{-1})=0$, where $k_4\: \BSp_4 \to K(\KMW_4,4)$ is the first nontrivial $k$-invariant for the Moore--Postnikov of \Cref{diag:MP-BSp}. \end{rmk}

We conclude this section with a cohomological classification of oriented vector bundles of rank $2$ over a smooth affine fourfold that are stably isomorphic to a fixed oriented bundle of rank $2$. For any integer $r \geq 0$, recall that $\Vect_r^{\mathcal{O}_{X}}(X)$ denotes the set of isomorphism classes of oriented vector bundles of rank $r$ over $X$. As for ordinary algebraic vector bundles, we have stabilization maps
\[
    \Phi^{o}_{r}\: \Vect_r^{\mathcal{O}_{X}}(X) \to \Vect_{r+1}^{\mathcal{O}_{X}}(X)
\]
given on representatives by 
\[(\mathcal{E}, \varphi) \mapsto (\mathcal{E} \oplus \mathcal{O}_{X}, \varphi^{+})\]
where $\varphi^{+}$ denotes the orientation $\varphi^{+}\: \det (\mathcal{E} \oplus \mathcal{O}_{X}) \xrightarrow{\cong} \det (\mathcal{E}) \xrightarrow{\varphi} \mathcal{O}_{X}$. We study the fibers ${(\Phi^o_r)}^{-1}[\mathcal{E} \oplus \mathcal{O}_{X}, \varphi^+]$ for some fixed oriented vector bundle $(\mathcal{E},\varphi)$ of rank $r$. Since $X$ is affine, we can describe this fiber in terms of commutative algebra.

\begin{definition} Let $R = \mathcal{O}_{X}(X)$ be the ring of regular functions on $X$ and let $(P, \theta)$ be the oriented projective $R$-module of rank $r$ corresponding to $(\mathcal{E},\varphi)$. Then let $\Um (P \oplus R)$ denote the set of surjective $R$-linear homomorphisms $P \oplus R \to R$. \end{definition}
Note that the group $\SL (P \oplus R)$ of $R$-linear automorphisms of $P \oplus R$ with determinant $1$ acts on the right on the set $\Um (P \oplus R)$ by precomposition. We denote by $\Um (P \oplus R)/\SL (P \oplus R)$ the corresponding orbit space. Following \cite[Section 2.A]{Sy1}, we consider the following construction.
\begin{construction}\label{const-bij-1}If $a \in \Um (P \oplus R)$, then its kernel $P_{a} = \ker (a)$ is a finitely generated projective $R$-module of rank $r$. Any section $s\colon R \to P \oplus R$ of $a$ then induces an isomorphism $i_{s}\colon P \oplus R \to P_{a} \oplus R$ 
and also an isomorphism $\det (P) \cong \det (P_{a})$ which is independent of the choice of $s$. 
We obtain an orientation of $\det(P_{a})$ by $\theta_{a}: \det (P_{a}) \xrightarrow{\cong} \det (P) \xrightarrow{\theta} R$. 
\end{construction}
As in \cite[Section 2.A]{Sy1}, \Cref{const-bij-1} induces a natural bijection
\[
    \Um(P \oplus R)/\SL (P \oplus R) \xrightarrow{\cong} {(\Phi^o_r)}^{-1}[\mathcal{E} \oplus \mathcal{O}_{X}, \varphi^+].
\]
Under the bijection above, the class of $a \in \Um (P \oplus R)$ is mapped to the class of the oriented vector bundle of rank $r$ over $X$ corresponding to the oriented projective $R$-module $(P_{a}, \theta_{a})$ of rank $r$.

We now focus on the case $r=2$ and we fix an oriented projective $R$-module $(P,\theta)$ of rank $2$. Recall that there is an abelian group $\tilde{V}_{\SL}(R)$ generated by triples of the form $[Q, \chi_{1}, \chi_{2}]$, where $\chi_{1}, \chi_{2}$ are non-degenerate alternating forms on some finitely generated projective $R$-module $Q$; the group $\tilde{V}_{\SL}(R)$ is isomorphic to the group $W_{\SL}(R)$ (cf. \cite[Section 2.C]{Sy1}), which is itself isomorphic to $\ker (K_{0}\Sp(R) \to K_{0}(R))$. Following \cite[Section 3.A]{Sy1}, there exists a well-defined map
\[
V_{\theta}\colon \Um (P \oplus R)/\SL (P \oplus R) \to \tilde{V}_{\SL}(R)
\]
called the generalized Vaserstein symbol modulo SL associated to $(P, \theta)$. 

The map $V_{\theta}$ above is defined as follows. Let $a \in \Um (P \oplus R)$ be an $R$-linear epimorphism and $P_{a} = \ker (a)$ the associated finitely generated projective $R$-module of rank $2$. The trivialization $\theta$ of $\det (P)$ induces a canonical non-degenerate alternating form $\chi$ on $P$ and a non-degenerate alternating form $\chi_{a}$ on $P_{a}$. As above, any section $s\colon R \to P \oplus R$ of $a$ induces an isomorphism $i_{s}\colon P \oplus R \to P_{a} \oplus R$. The triple
\[
    V_{\theta}(a) = [P \oplus R^2, \chi \perp \psi_{2}, {(i_{s} \oplus 1)}^{t} (\chi_{a} \perp \psi_{2}) {(i_{s} \oplus 1)}] \in \tilde{V}_{\SL}(R)
\]
is independent of the choice of $s$ and induces the generalized Vaserstein symbol modulo SL associated to $(P, \theta)$ above. We refer the reader to \cite[Section 3.A]{Sy1} for details. Injectivity and surjectivity of $V_{\theta}$ was studied in detail in \cite[Section 3.B]{Sy1}. The following theorem is now an easy consequence of \Cref{thm:sp-inj}:
\begin{theorem}\label{thm:Vaserstein}
Let $R$ be a smooth affine domain of dimension $4$ over an algebraically closed field $k$ of characteristic not equal to $2$ or $3$. Let $(P,\theta)$ be an oriented projective $R$-module of rank $2$ over $X$. Then the generalized Vaserstein symbol
\[
    V_{\theta}\colon \Um (P \oplus R)/\SL (P \oplus R) \to \tilde{V}_{\SL}(R)
\]
is bijective.
\end{theorem}
\begin{proof}
The map $V_{\theta}$ is automatically surjective by \cite[Theorem 3.2]{Sy1} and Suslin's cancellation theorem \cite{Su1}. So it remains to show that $V_{\theta}$ is injective.

For this purpose, assume $a, a' \in \Um (P \oplus R)$ with sections $s,s'\colon R \to P \oplus R$ such that $V_{\theta}(a) = V_{\theta}(a')$. Then, following the formalism developed in \cite[Lemma 3.4]{Sy1}, we easily see that the non-degenerate alternating forms ${(i_{s} \oplus 1)}^{t} (\chi_{a} \perp \psi_{2}) {(i_{s} \oplus 1)}$ and ${(i_{s'} \oplus 1)}^{t} (\chi_{a'} \perp \psi_{2}) {(i_{s'} \oplus 1)}$ are isometric. In other words, the symplectic $R$-modules $(P \oplus R^2, {(i_{s} \oplus 1)}^{t} (\chi_{a} \perp \psi_{2}) {(i_{s} \oplus 1)}) \cong (P_{a} \oplus R^2, \chi_{a} \perp \psi_2)$ and $(P \oplus R^2, {(i_{s'} \oplus 1)}^{t} (\chi_{a'} \perp \psi_{2}) {(i_{s'} \oplus 1)}) \cong (P_{a'} \oplus R^2, \chi_{a'} \perp \psi_{2})$ are isomorphic. Now, by \Cref{thm:sp-inj}, it follows that $(P_{a}, \chi_{a})$ and $(P_{a'}, \chi_{a'})$ are isomorphic as symplectic and hence as oriented $R$-modules. Therefore, by the preceding paragraphs, the classes of $a$ and $a'$ are the same in the orbit space $\Um (P \oplus R)/\SL (P \oplus R)$ and $V_{\theta}$ is indeed injective.
\end{proof}
\begin{rmk} The special case of \Cref{thm:Vaserstein} when $P = R^2$ was already proven in \cite[Theorem 3.9]{Sy3} and was significantly easier.
\end{rmk}  \Cref{thm:Vaserstein} now implies the following cohomological description of 
${(\Phi^o_2)}^{-1}[\mathcal{E} \oplus \mathcal{O}_{X}, \varphi^+]$:
\begin{theorem}\label{thm:oriented-fibers}
Let $X$ be a smooth affine variety of dimension $4$ over an algebraically closed field $k$ of characteristic not equal to $2$ or $3$. Let $(\mathcal{E},\varphi)$ be an oriented vector bundle of rank $2$ over $X$. Then there is a natural bijection
\[
{(\Phi^o_2)}^{-1}[\mathcal{E} \oplus \mathcal{O}_{X}, \varphi^+ ]\cong \ker \left(\KSp_{0}(X) \to \mathbf{K}_{0}(X) \right).
\]
\end{theorem}

\section{The number of algebraic vector bundles with a fixed topological class}

Let $X$ be a smooth affine variety of dimension $d$ over $\C$ and let $X^\mathrm{an}$ denote its associated analytic space (i.e., $X^\mathrm{an}=X(\mathbb{C})$ viewed as a complex manifold). Then complex realization induces natural maps
\[\mathfrak{R}_{r}(X):\mathcal{V}_{r}(X) \to \Vect^{\top}_{r}(X^\mathrm{an})\]
for all integers $r \geq 0$, where $\mathcal{V}_{r}(X)$ denotes the set of isomorphism classes of algebraic vector bundles of rank $r$ over $X$ and $\Vect^{\top}_{r}(X^\mathrm{an})$ denotes the set of isomorphism classes of complex topological vector bundles of rank $r$ over $X^{\mathrm{an}}$. Similarly, complex realization induces cycle class maps
\[ \cl_{i} : \CH^{i}(X) \to H^{2i}_{\sing}(X^\mathrm{an},\Z)\]
for all integers $i \geq 0$, where $H_{\sing}^i(-,\Z)$ denotes singular cohomology with integer coefficients. The maps $\mathfrak{R}_{r}(X)$ are not surjective in general. Even isomorphism classes of complex topological vector bundles of rank $r$ whose topological Chern classes lie in the image of the cycle class maps need not lie in the image of the map $\mathfrak{R}_{r}(X)$ (cf. \cite[Theorem 2]{AFH-alg19}).

It is natural to study the fibers of the maps $\mathfrak{R}_{r}(X)$, i.e., to determine the set of isomorphism classes of algebraic vector bundles of rank $r$ over $X$ realizing to the same isomorphism class of a complex topological vector bundle over $X^{\mathrm{an}}$. If $r=1$, then it is classical that topological and algebraic vector bundles are both uniquely determined by their first Chern classes (appropriately interpreted) and we have a commutative diagram
\[\begin{tikzcd}
    \mathcal{V}_{1}(X)\ar[d, "\mathfrak{R}_{1}(X)"]\ar[r, "{c_{1}}" above, "\cong" below] & \CH^1(X)\ar[d, "\cl_{1}"]\\
    \Vect^{\top}_{1} (X^{\mathrm{an}})\ar[r, "{c_{1}^{\top}}" above,"\cong" below] & H_{\sing}^2(X^{\mathrm{an}},\Z),
\end{tikzcd}\]
This immediately implies the following theorem:
\begin{proposition}
Let $X$ be a smooth affine variety of dimension $d$ over $\C$. Then there is a bijection between any non-empty fiber of the map $\mathfrak{R}_{1}(X)$ and the kernel of the cycle class map $\cl_{1}$. In particular, the non-empty fibers of the map $\mathfrak{R}_{1}(X)$ are all finite (resp. a singleton) if and only if the kernel of the map $\cl_{1}$ is finite (resp. trivial).
\end{proposition}
If $r=2$ and $d \leq 3$, then by the Andreotti--Frankel theorem \cite[Theorem 1]{andreotti1959lefschetz} and  \cite[Proposition 2.29]{AntieauElmanto}, $c_1^{\top}$ is bijective. By \cite[Theorem 1]{AF-3fold}, $(c_1,c_2)$ is also bijective. Thus, we have a commutative diagram:
\[\begin{tikzcd}
    \mathcal{V}_{2}(X)\ar[d, "\mathfrak{R}_{2}(X)" left]\ar[r, "{(c_{1},c_{2})}" above, "\cong" below] & \CH^1(X) \times \CH^{2}(X)\ar[d, "\cl_{1} \circ \mathrm{pr}_{\CH^{1}(X)}"]\\
    \Vect^{\top}_{2} (X^{\mathrm{an}})\ar[r, "{c_{1}^{\top}}" above,"\cong" below] & H_{\sing}^2(X^{\mathrm{an}},\Z).
\end{tikzcd}\]
Note that the group $\CH^{2}(X)$ is trivial if $d=1$. The diagram immediately yields the following result:
\begin{proposition}
Let $X$ be a smooth affine variety of dimension $d \leq 3$ over $\C$. Then there is a bijection between any non-empty fiber of the map $\mathfrak{R}_{2}(X)$ and the product $\ker (\cl_{1})\times \CH^{2}(X)$. In particular, the non-empty fibers of the map $\mathfrak{R}_{2}(X)$ are all finite (resp. a singleton) if and only if $\ker(cl_{1})$ and $\CH^{2}(X)$ are finite (resp. trivial).
\end{proposition}
Now let $r=2$ and $d=4$. Then \cite[Theorem 1]{andreotti1959lefschetz} and \cite[Proposition 2.29]{AntieauElmanto} still apply and the map $(c_{1}^{\top},c_{2}^{\top})$ is bijective, but the map $(c_{1},c_{2})$ is no longer injective in general. We obtain a commutative diagram of the form
\begin{equation}\label[diagram]{diag:top-alg-rk-2(2)} \begin{tikzcd}[column sep=4em]
    \Vect_2(X)\ar[d, "\mathfrak{R}_{2}(X)"]\ar[r, "{(c_{1},c_{2})}" above] & \CH^1(X) \times \CH^2(X)\ar[d, "cl_{1} \times cl_{2}"]\\
    \Vect^{\top}_2(X(\C))\ar[r, "{(c_{1}^{\top},c_{2}^{\top})}" above,"\cong" below] & {H_{\sing}^2(X,\Z) \times H_{\sing}^4(X,\Z)},
\end{tikzcd} \end{equation}
Although $(c_{1},c_{2})$ is not bijective in general, the map $\mathfrak{R}_{2}(X)$ can be understood in terms of $(c_{1},c_{2})$ and $cl_{1} \times cl_{2}$. Using \Cref{theorem-1}, we can deduce the following theorem on the fibers of $\mathfrak{R}_{2}(X)$ from \Cref{diag:top-alg-rk-2(2)}:

\begin{theorem} Let $X$ be a smooth affine variety of dimension $4$ over $\C$. If $\ker (cl_{1})$, $\ker (cl_{2})$, $\Ch^3(X)$ and $\Hmot^{5}(X, \Z/2(3))$ are finite abelian groups (resp. trivial), then every non-empty fiber of the map $\mathfrak{R}_{2}(X)$ is finite (resp. a singleton).

If $\Ch^3(X)$ and $\Hmot^5(X,\Z/2(3))$ are zero, then there is a bijection between any non-empty fiber of $\mathfrak{R}_{2}(X)$ and $\ker(\cl_1)\times \ker(\cl_2)$. 
\end{theorem}

\section{Examples}\label{sec:examples}
On a given smooth affine fourfold $X$ over an algebraically closed field of characteristic not equal to $2$ or $3$, \Cref{theorem-1} reduces the study of rank $2$ algebraic vector bundles with fixed Chern classes to understanding $\Ch^3(X)$ and $\Hnr^4(X,\mu_2)$. If we instead study vector bundles with prescribed first Chern class and Euler class, the answers take a slightly different form. Below, we use  \Cref{thm:ch3-zero-determined-euler} to give a number of explicit cases where rank $2$ vector bundles on $X$ are completely classified by their first Chern class and Euler class. We also give a number of examples where there are only finitely many vector bundles with prescribed Chern classes or first Chern class and Euler class.

As a preliminary remark, we note that a number of concrete examples are provided by work of Asok--Fasel in the case that a given smooth affine variety has Nisnevich cohomological dimension at most $3$. For example, let $X$ be a smooth affine fourfold over an algebraically closed field of characteristic not $2$ having the $\A^1$-homotopy type of a smooth surface (e.g., if $X$ is the Jouanolou device over any smooth projective surface, cf.~\cite[3.3.6]{AF-survey}) or of a smooth affine threefold. Then any rank $2$ vector bundle over $X$ is uniquely determined by its first two Chern classes \cite[Theorem 1]{AF-3fold}. 

The remainder of this section will focus on examples that are not covered by previous literature. In \Cref{subsec:ex-hypersurface} we focus on complements of hypersurfaces in $\P^4$,  $\P^1 \times \P^3$, and $\P^2 \times \P^2$. In \Cref{ex:MK}, we discuss the Mohan Kumar example in dimension $4$. In \Cref{ex:cyclic}, we study rank $2$ algebraic vector bundles on cyclic coverings.

Throughout this section, let $k$ be an algebraically closed base field of characteristic different from $2$ or $3$.  We will further specialize the field as necessary.

\subsection{Rank $2$ vector bundles on hypersurface complements}\label{subsec:ex-hypersurface}
The examples that follow are all smooth affine fourfolds of the form $Y \setminus D$ for $Y$ a smooth projective variety of dimension $4$ and $D$ a hypersurface. We organize them according to the ambient projective variety $Y$.

\begin{example}[Hypersurface complements in $\P^4$]\label{ex:p4}  Consider $X = \P^4 \minus D$ where $D$ is a hypersurface of degree $d$. From the localization sequence on Chow groups, $\CH^3(X)$ is a quotient of $\Z/d$, and hence $\Ch^3(X)$ is finite. By \Cref{rmk:fibers-ch3-finite}, there are only finitely many vector bundles with a given first Chern class and Euler class. We can consider a number of additional hypotheses that allow us to draw stronger conclusions:
\begin{enumerate}
\item If $d$ is odd, $\Ch^3(X)=0$ and rank $2$ vector bundles on $X$ are uniquely determined by their first Chern class and Euler class by \Cref{thm:ch3-zero-determined-euler}.
\item If $d\le 5$, then $\Ch^3(X)$ is zero. This follows from observing that the Fano variety of lines on $D$ is nonempty \cite[Theorem~8]{Barth-VandeVen}. Again, rank $2$ vector bundles on $X$ are determined by their first Chern class and Euler class.
\item Appealing to \Cref{prop:nr-compactify-text}, if $D$ is smooth and $H^3_\nr(D,\mu_2^{3})$ is finite (resp., zero), then $H_\nr^4(X,\mu_4^{\otimes4})$ is finite (resp., zero) as well. By \Cref{theorem-1} and \Cref{lem:stage1}, there are only finitely many rank $2$ bundles over $X$ with given Chern classes. 
\item By \Cref{exa:projective-hypersurface-complements-H4}, if $k=\mathbb C$, $D$ is smooth, and $d\leq 4$, then $\Hnr^4(X,\mu_2)=0$ and the previous item implies Chern finiteness for rank $2$ bundles on $X$, i.e., there are only finitely many rank $2$ bundles over $X$ with given Chern classes.
\item Let $k={\mathbb C}$ and $X=\mathbb P^4 \setminus D$ for $D$ a smooth hypersurface of degree $\leq 4$. Then it follows from (a) and \cite[Theorem 2.2.2]{AFH-alg19} that any pair $(c_{1},c_{2}) \in \CH^1 (X) \times \CH^{2}(X)$ can be realized as the first two Chern classes of an algebraic vector bundle of rank $2$ over $X$. Combining (b) and (c) above with \Cref{theorem-1} and \Cref{lem:stage1}, we then find that vector bundles of rank $2$ over $X$ are completely classified by their Chern classes.
\end{enumerate}
\end{example}

We highlight a strong consequence of the previous example, part (e):

\begin{theorem}\label{thm:p4-complement-threefold-fourfold} 
Suppose $k= \mathbb C$ and $X$ is the complement of a smooth hypersurface $D$ of degree $d \leq 4$ in $\P^4$. Then there are exactly $d^2$ isomorphism classes of algebraic vector bundles of rank $2$ over $X$, determined uniquely by their first two Chern classes in $\CH^1 (X) \times \CH^2 (X) \cong (\Z/d)^{\times 2}$.
\end{theorem}

\begin{question} What can be said about the vanishing or finiteness of $H^3_\nr(X,\mu_2)$ when $X$ is the complement of a generic or a specific quintic threefold?
\end{question}

\begin{example}[Hypersurface complements in $\P^1 \times \P^3$]\label{ex:p1p3} Consider $X=\P^1 \times \P^3 \setminus D$, where $D$ is a hypersurface of bidegree $(a,b)$. Note that $\CH^3(X)$ is a quotient of $\Z/g \oplus \Z/\frac{b^{2}}{g}$, where $g = \op{gcd}(a,b)$. Hence $\Ch^3(X)$ is finite. By \Cref{rmk:fibers-ch3-finite}, there are only finitely many vector bundles with a given first Chern class and Euler class. Again, we consider hypotheses under which we may conclude more:
\begin{enumerate}
\item If $b$ is odd, then $\Ch^3(X)$ is trivial and by \Cref{thm:ch3-zero-determined-euler} vector bundles on $X$ are determined by their first Chern class and Euler class.
\item By \Cref{prop:nr-compactify-text}, if $D$ is smooth and $H^3_\nr(D,\mu_2)$ is finite (resp., zero) then $H_\nr^4(X,\mu_4)$ is finite (resp. zero) as well. Therefore, by \Cref{theorem-1} and \Cref{lem:stage1}, there are only finitely many rank $2$ bundles over $X$ with given Chern classes. 
\item If $k=\mathbb C$ and $D$ is smooth, \Cref{exa:projective-hypersurface-complements-H4} applies: if $a \leq 1$ and $b\leq 3$, $-K_D$ is ample and hence $\Hnr^4(X,\mu_2)=0$. In particular, the number of isomorphism classes of rank $2$ bundles over $X$ with given Chern classes is finite.
\item Consider the specific case of $D$ a smooth bidegree $(1,3)$ hypersurface over $k=\mathbb C$. Then it follows from (a) and \Cref{lem:every-class-is-euler-class-of-bundle} that any pair $(c_{1},c_{2}) \in \CH^1 (X) \times \CH^{2}(X)$ can be realized as the first two Chern classes of an algebraic vector bundle of rank $2$ over $X$. Combining (a) and (c) of this example with \Cref{theorem-1} and \Cref{lem:stage1}, we find that isomorphism classes of rank $2$ vector bundles over $X$ are in bijection with choices $(c_1,c_2) \in \CH^1(X) \times \CH^2(X)$. 
\end{enumerate}
\end{example}

\begin{example}[Hypersurface complements in $\P^2 \times \P^2$]\label{ex:p2p2} Consider $X=\P^2 \times \P^2 \setminus D$, where $D$ is a smooth hypersurface of bidgree $(a,b)$. Let $g = gcd(a,b)$ denote the greatest common divisor of $a$ and $b$. Then $\CH^3(X)$ is a finite abelian group which is $g^2$-torsion. To see this, let $x,y$ be generators of $\CH^1(X)$ such that $ax+by=0$. Note also that we may choose $x$ and $y$ so that $x^3=y^3=0$ and $\CH^3(X)$ is generated by $xy^2$ and $yx^2$. Moreover, $axy^2=0$ and $byx^2=0$. We can write $g=an+bm$ for some $m,n \in \Z$ and we find that:
$$g^2(xy^2)=b^2m^2xy^2=m^2(by)^2x= m^2(ax)^2 x= m^2 a^2 x^3=0$$ and similarly we deduce that $g^2(yx^2)=0$. 

In any case, $\CH^3(X)$ is finite and there are only finitely many vector bundles with a given first Chern class and Euler class. However, we can sometimes say more:
\begin{enumerate}
\item If the greatest common divisor of $a$ and $b$ is odd, then $\Ch^3(X)$ is trivial and by \Cref{thm:ch3-zero-determined-euler} vector bundles on $X$ are determined by their first Chern class and Euler class.
\item By \Cref{prop:nr-compactify-text}, if $D$ is smooth and $H^3_\nr(D,\mu_2^{3})$ is finite (resp., zero), then $H_\nr^4(X,\mu_4^{\otimes4})$ is finite (resp., zero) as well. Therefore, by \Cref{theorem-1} and \Cref{lem:stage1}, there are only finitely many rank $2$ bundles over $X$ with given Chern classes.
\item If $k=\mathbb C$, \Cref{exa:projective-hypersurface-complements-H4} applies. For $D$ smooth and $a,b \leq 2$, we find that $-K_X$ is ample and hence $\Hnr^4(X,\mu_2)$ is trivial. In particular, there are only finitely many vector bundles of rank $2$ with fixed Chern classes.
\item Consider the specific case $k=\mathbb C$, $D$ smooth, and $a=1$ and $b=2$. Then it follows from (a) and \Cref{lem:every-class-is-euler-class-of-bundle} that any pair $(c_{1},c_{2}) \in \CH^1 (X) \times \CH^{2}(X)$ can be realized as the first two Chern classes of an algebraic vector bundle of rank $2$ over $X$. Combining (a) and (c) of this example with \Cref{theorem-1} and \Cref{lem:stage1}, we find that isomorphism classes of rank $2$ vector bundles over $X$ are in bijection with choices $(c_1,c_2) \in \CH^1(X) \times \CH^2(X)$.
 \end{enumerate}
\end{example}

\begin{example}[The Mohan Kumar example]\label{ex:MK}
In \cite{MK1} and \cite{MK}, N. Mohan Kumar gives an amazing construction of a smooth complex fourfold with non-trivial but stably trivial rank $2$ vector bundle. In fact, Mohan Kumar's construction applies more generally to give smooth affine $(p+2)$-folds with non-trivial but stably trivial rank $p$ bundles for any prime $p$, and his construction works over any algebraically closed base field. Some additional perspective on such examples is provided in \cite{WendtMK21} using motivic methods.

In the case $p=2$, the variety in question is the complement of a union of three hypersurfaces in $\mathbb P^3 \times \A^1$. While Mohan Kumar's construction provides defining equations for these hypersurfaces, explicit computations are still extremely difficult and thus it is challenging to understand stably trivial vector bundles on the Mohan Kumar examples. It follows from \cite[Theorem 3.9]{Sy3} that non-trivial stably trivial oriented rank $2$ bundles over the Mohan Kumar variety $X_{\MK}$ are in bijection with the Hermitian $K$-theory group $W_{\SL}(X_{\MK}) \cong \ker (\KSp_{0}(X_{\MK}) \to \mathbf{K}_{0}(X_{\MK}))$; this was done by studying classical Hermitian $K$-theory groups. \Cref{thm:oriented-fibers} further generalizes this and shows that all the fibers $({\Phi_{2}^{o}})^{-1}([\mathcal{E},\varphi])$ for an oriented vector bundle $(\mathcal{E},\varphi)$ of rank $2$ over $X_{\MK}$ are in bijection with the group $W_{\SL}(X_{\MK})$.
\Cref{thm:sp-inj} shows that oriented vector bundles of rank $2$ over $X_{\MK}$ are determined by their symplectic $K$-theory class. More generally, \Cref{symp-classification} implies that vector bundles $\mathcal E$ on $X_{\MK}$ with fixed determinant $\mathcal N^{\otimes 2}$ are determined by the symplectic $K$-theory class of $\mathcal E \otimes \mathcal N^{-1}$. 
\end{example}

\subsection{Rank $2$ vector bundles over cyclic coverings}\label{ex:cyclic}
Let $X$ be a smooth affine variety over a field $k$ which admits an embedding $\iota\colon k \hookrightarrow \C$. Then one may use this embedding to define an associated complex manifold $X_{\iota}^{\mathrm{an}}$. The variety $X$ is called topologically contractible if the complex manifold $X_{\iota}^{\mathrm{an}}$ is a contractible topological space for every embedding $\iota\colon k \hookrightarrow \C$. Affine spaces over $k$ are the primordial examples of topologically contractible smooth affine $k$-varieties. The study of topologically contractible varieties has stimulated a wealth of research over many decades and its long history is directly related to important questions in algebraic geometry such as the Zariski cancellation problem, the linearization problem, the generalized van de Ven question or the generalized Serre question on algebraic vector bundles \cite[Section 5.1.2]{AsokOestvaer}.

Answering a question raised by J.-P. Serre (cf. \cite[p. 243]{S1}), D. Quillen and A. Suslin independently proved that algebraic vector bundles over affine spaces are always trivial \cite{Q}, \cite{Su2}. The generalized Serre question asks whether algebraic vector bundles over topologically contractible smooth affine complex varieties are always trivial \cite[Question 6]{AsokOestvaer}. While the generalized Serre question is known to have a positive answer in dimensions $\leq 2$, the question remains completely open in higher dimensions \cite[Section 5.5.2]{AsokOestvaer}.

Many concrete examples of topologically contractible smooth affine complex varieties in the literature are given by cyclic coverings \cite[Section 5]{Z}, defined as follow. Let $X$ be a smooth affine variety over an algebraically closed field $k$ of characeristic $0$, $n>0$ be an integer and $f \in \mathcal{O}_{X}(X)$ be a regular function. In general, one assumes that
\begin{itemize}
\item the closed subscheme $F_0$ of $X$ defined by $f$ is smooth over $k$,
\item the polynomial $u^n -f$ is prime in both the polynomial rings $\mathcal{O}_{X}(X)[u]$ and $k(X)[u]$, where $k(X)$ is the field of fractions of the domain $\mathcal{O}_{X}(X)$, and
\item there is a $\mathbb{G}_m$-action on $X$ which makes $f$ a quasi-invariant of weight $d \in \mathbb{Z}$ with $\langle d,n \rangle = \mathbb{Z}$.
\end{itemize}
Under these general assumptions, one obtains a smooth affine $k$-variety defined by \[Y_{n}:= \{ u^n - f = 0\} \subset X \times \mathbb{A}^1,\] where $u$ is the variable of $\mathbb{A}^1$. The projection morphism $\varphi_{n}\colon Y_{n} \rightarrow X$ is called a \emph{cyclic covering} of $X$ of order $n$ with respect to $f$. As indicated above, many examples of topologically contractible smooth affine complex varieties can be constructed as cyclic coverings; indeed, if $k=\mathbb{C}$, if the variety $X$ is topologically contractible and if other natural assumptions are satisfied, the smooth affine $k$-variety $Y_{n}$ is topologically contractible \cite[Theorem 5.1]{Z}.

The motivic cohomology groups of cyclic coverings over algebraically closed fields of characteristic $0$ were studied in \cite{Sy2}. Combining \cite[Theorem 3.21]{Sy2} and \cite[Corollary 3.25]{Sy2} for $\Z/2$-coefficients, one obtains that if $n$ is odd the cyclic covering morphism $\varphi_{n}$ induces isomorphisms \[\CH^{i}(X) \otimes_{\Z} \Z/2 \cong \CH^{i}(Y_{n}) \otimes_{\Z} \Z/2\] for $i \geq 0$. If furthermore $\CH^{2}(X) \otimes_{\Z} \Z/2 = 0$, then $\varphi_{n}$ also induces an isomorphism \[\Hmot^{5}(X, \Z/2(3)) \cong \Hmot^{5}(Y_{n}, \Z/2(3)).\] In particular, it follows that whenever $\CH^{i}(X) \otimes_{\Z} \Z/2 = 0$ for all $i \geq 1$ and $\Hmot^{5}(X, \Z/2(3)) = 0$ holds, the same will hold for $Y_{n}$. We give some concrete examples:
\begin{enumerate}
\item For example, consider $\alpha_{0}, \alpha_{1}, \alpha_{2}, \alpha_{3} \geq 2$ are pairwise coprime integers with $\alpha_{0}$ odd. Let $n \geq 2$ be an integer coprime to $\alpha_{0}$ and let $$\alpha_{3} = 1 + m \alpha_{1} \alpha_{2}$$ for some $m > 0$. We obtain a smooth affine subvariety 
\[Y_{\alpha_{0}} =  \{ x + x^n z^{\alpha_{0}} + y_{1}^{\alpha_{1}} + y_{2}^{\alpha_{2}} + y_{3}^{\alpha_{3}} = 0 \} \subset \A^5.\] 
Let \[X=\{x + x^n z + y_{1}^{\alpha_{1}} + y_{2}^{\alpha_{2}} + y_{3}^{\alpha_{3}} =0\} \subset \A^5.\] The  variety $Y$ is a cyclic covering of $X$
of order $\alpha_{0}$ along the function $z \in \mathcal{O}_{X}(X)$. The variety $X$ is stably $\A^1$-contractible by \cite[Theorem 1.19]{DPO} and hence has the motivic cohomology of $\Spec(k)$. In particular, $\CH^{i}(X) \otimes_{\Z} \Z/2 = 0$ for all $i \geq 1$ and $\Hmot^{5}(X, \Z/2(3)) = 0$. By the preceding paragraph, one therefore also has $\CH^{i}(Y_{\alpha_{0}}) \otimes_{\Z} \Z/2 = 0$ for all $i \geq 1$ and $\Hmot^{5}(Y_{\alpha_{0}}, \Z/2(3)) = 0$. 
As a consequence, we conclude that algebraic vector bundles of rank $2$ over $Y_{\alpha_{0}}$ are uniquely determined up to isomorphism by their Chern classes. If $k = \C$, the variety $Y_{\alpha_{0}}$ is topologically contractible by \cite[Example 6.2]{Z}. We refer the reader to \cite[Section 4.3]{Sy2} for details.

\item Let $\alpha_{0}, \alpha_{1}, \alpha_{2}, \alpha_{3} \geq 2$ are pairwise coprime integers. Consider the smooth affine subvariety
\[Y =\{ x + x^2 {(u^{\alpha_{0}} + v^{\alpha_{1}})} + z^{\alpha_{2}} + t^{\alpha_{3}} = 0\} \subset \A^5,\]  and the subariety
 \[X=\{x+x^2 y+ z^{\alpha_{2}} + t^{\alpha_{3}} = 0\} \subset \A^4.\] 
Then $Y$ is a cyclic covering of the product of $X\times \A^1$ of order $\alpha_{0}$ with respect to the function $y - v^{\alpha_{1}} \in \mathcal{O}_{X}(X)[v],$ where we consider $v$ as the coordinate on $\A^1$. Analogously, $Y$ is also a cyclic covering of the product of $X\times \A^1$ of order $\alpha_{1}$ with respect to the function $y - u^{\alpha_{0}} \in \mathcal{O}_{X}(X)[u]$, where we consider $u$ as the coordinate on $\A^1$.. The variety $X$ is a Koras--Russell threefold of the first kind and is $\A^1$-contractible by \cite[Theorem 1]{DubFas}. In particular, $X$ as well as $X \times \A^1$ have the motivic cohomology of $\Spec(k)$. As either $\alpha_{0}$ or $\alpha_{1}$ is odd, it follows that $\CH^{i}(Y) \otimes_{\Z} \Z/2 = 0$ for all $i \geq 1$ and $\Hmot^{5}(Y, \Z/2(3)) = 0$. As a consequence, algebraic vector bundles of rank $2$ over $Y$ are uniquely determined up to isomorphism by their Chern classes. As a matter of fact, one can even prove that $\CH^{i}(Y) = 0$ for $i \geq 1$; together with the vanishing of $\Hmot^{5}(Y, \Z/2(3))$ mentioned above, this implies that all vector bundles over $Y$ are actually trivial! If $k = \C$, the variety $Y$ is topologically contractible by \cite[Example 6.2]{Z}. We refer the reader to \cite[Section 4.4]{Sy2} for details.  
\end{enumerate}

\bibliographystyle{amsalpha}
\bibliography{twists}{}

\end{document}